\documentclass[reqno,oneside,12pt]{amsart}

\usepackage[T1]{fontenc}
\usepackage{times,mathptm}
\usepackage{amssymb,epsfig,verbatim,xypic,xcolor, graphicx, enumerate, tikzsymbols, transparent}

\theoremstyle{plain}

\newtheorem{thm}{Theorem}[section]
\newtheorem{cor}[thm]{Corollary}
\newtheorem{pro}[thm]{Proposition}
\newtheorem{lem}[thm]{Lemma}
\newtheorem{proposition-principale}[thm]{Proposition principale}
\newtheorem{thm-principal}{Main Theorem}

\theoremstyle{definition}

\newtheorem{eg}[thm]{Example}
\newtheorem{rem}[thm]{Remark}

\newenvironment{defi-G}
{\noindent{\bf Definition.}\it}{\\}

\newenvironment{thm-M}
{\noindent{\bf Main Theorem.}\it }{}

\newenvironment{thm-C}
{\noindent{\bf Classification Theorem.}\it }{}

\newenvironment{thm-AA}
{\noindent{\bf Theorem A'.}\it}{\\ }

\newenvironment{thm-B}
{\noindent{\bf Theorem B.}\it}{\\ }

\newenvironment{thm-BB}
{\noindent{\bf Theorem B'.}\it}

\def\C{\mathbb{C}}
\def\R{\mathbb{R}}
\def\Q{\mathbb{Q}}
\def\H{\mathbb{H}}
\def\Z{\mathbb{Z}}
\def\N{\mathbb{N}}

\def\ii{{\sf{i}}}

\def\T{{\sf{T}}}

\def\P{\mathbb{P}}
\def\Sphere{\mathbb{S}}
\def\disk{\mathbb{D}}
\def\A{\mathbb{A}}

\def\SC{S(\mathbb C)}

\def\Aut{{\sf{Aut}}}
\def\Out{{\sf{Out}}}
\def\Sym{{\sf{Sym}}}
\def\Homeo{{\sf{Homeo}}}
\def\Graph{{\mathcal{G}}}

\def\Fam{{\sf{Fam}}}
\def\Rep{{\sf{Rep}}}
\def\Mod{{\sf{Mod}}}

\def\Sing{{\sf{Sing}}}

\def\Lim{{\mathrm{Lim}}}

\def\area{{\mathrm{Area}}}

\def\Supp{{\rm{Supp}}}
\def\Orb{{\rm{Orb}}}
\def\Init{{\mathrm{In}}}

\def\PGL{{\sf{PGL}}\,}
\def\PSL{{\sf{PSL}}\,}
\def\GL{{\sf{GL}}\,}

\def\SU{{\sf{SU}}\,}
 
\def\SL{{\sf{SL}}\,}

\def\Mat{{\sf{Mat}}\,}

\def\SU{{\sf{SU}}\,}

\def\tr{{\sf{tr}}}

\newcommand{\Id}{{\rm Id}}

\def\Ind{{\text{Ind}}}

\def\Sbar{{\overline S}}

\def\Sym{{\text{Sym}}}

\def\dist{{\sf{dist}}}

\def\dd{{\mathrm{d}}}

\newcommand{\christophe}[1]{{\color{blue}*}\marginpar{\tiny  \color{blue} CD: #1}}

\newcommand{\serge}[1]{{\color{red}*}\marginpar{\tiny  \color{red} SC: #1}}

\addtocounter{section}{0}             
\numberwithin{equation}{section}       

\begin{document}

\setlength{\baselineskip}{0.56cm}        
%
%
\title[Dynamics on Markov surfaces]
{Dynamics on Markov surfaces: classification of stationary measures}
\date{2023/2024}
\author{Serge Cantat, Christophe Dupont and Florestan Martin-Baillon}
\address{CNRS, IRMAR - UMR 6625 \\ 
Universit{\'e} de Rennes 
\\ France}
\email{serge.cantat@univ-rennes.fr}
\email{christophe.dupont@univ-rennes.fr}
\email{florestan.martin-baillon@univ-rennes.fr}
%
%

%
%

%
%

\begin{abstract}
Consider the four punctured sphere $\Sphere_4^2$. Each choice of four traces, one for each puncture, determines a relative character variety for the representations of the fundamental group of  $\Sphere_4^2$ in $\SL_2(\C)$. 
We classify the stationary probability measures for the action of the mapping class group $\Mod(\Sphere_4^2)$ on these character varieties.
\vspace{0.15cm}

\noindent{\sc{R\'esum\'e.}} Soit  $\Sphere_4^2$ la sphère privée de quatre points. \`A chaque choix de quatre traces, une par épointement, est associée une variété de caractères relative pour les représentations du groupe fondamental de  $\Sphere_4^2$ dans $\SL_2(\C)$. Nous classons les mesures de probabilité stationnaires pour l'action du grou\-pe modulaire $\Mod(\Sphere_4^2)$ sur ces variétés. 
\end{abstract}

\maketitle

\setcounter{tocdepth}{1}

\tableofcontents

\vfill




{\small{The research activities of the authors  are partially funded by the European Research Council (ERC GOAT 101053021). The authors benefited from the support of the French government "Investissements d'Avenir" program integrated to France 2030 (ANR-11-LABX-0020-01).}}

\pagebreak

\section{Introduction} 

\subsection{Representations} Let $\Sphere_4^2$ denote the $2$-dimensional sphere with four punctures. A presentation of its fundamental group is 
\begin{equation}
\pi_1(\Sphere_4^2)=\langle \alpha, \beta, \gamma, \delta \; \vert \; \alpha\beta \gamma \delta = 1 \rangle
\end{equation}
where the generators $\alpha$, $\beta$, $\gamma$, $\delta$ correspond to disjoint simple loops around the punctures (see~Figure~\ref{fig:sphere}, taken from~\cite{Cantat-Loray}). Thus, $\pi_1(\Sphere_4^2)$ is the free group on the three generators $\alpha$, $\beta$, and $\gamma$.

\begin{figure}[h]\label{fig:four_punctured_sphere}
\includegraphics[width=0.5\textwidth]{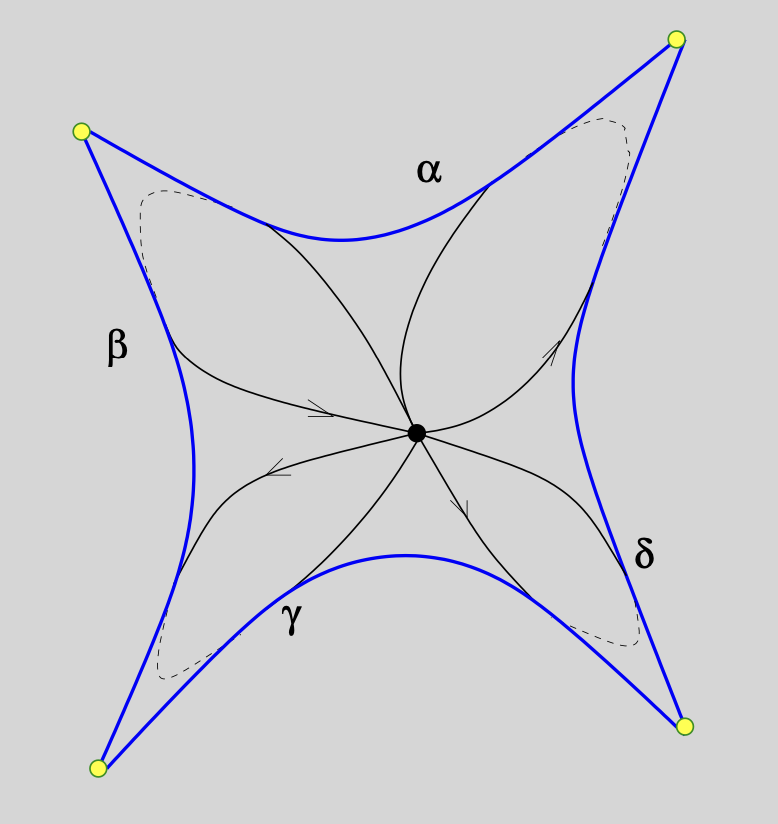}
\caption{ {\sf{The four punctured sphere. }}}
\label{fig:sphere}
\end{figure}

Let $\Rep(\Sphere_4^2)$ be the set of representations of $\pi_1(\Sphere_4^2)$ in $\SL_2(\C)$. Such a representation $\rho$ is uniquely determined by the three matrices $\rho(\alpha)$, $\rho(\beta)$, $\rho(\gamma)$, so that $\Rep(\Sphere_4^2)$ can be identified to $\SL_2(\C)^3$. As in~\cite{Benedetto-Goldman, Cantat-Loray}, we associate the following traces to such a representation
\begin{eqnarray}
a=\tr(\rho(\alpha)), \ b= \tr(\rho(\beta)), \ c=\tr(\rho(\gamma)), \ d=\tr(\rho(\delta)) \\
x=\tr(\rho(\alpha\beta)), \ y=\tr(\rho(\beta\gamma)), \ z=\tr(\rho(\gamma\alpha)). \quad\quad
\end{eqnarray}
Then, the polynomial map $\chi\colon \Rep(\Sphere_4^2)\to \A^7(\C)$ defined by 
\begin{equation}
\chi(\rho)=(a,b,c,d,x,y,z)
\end{equation} 
is invariant under conjugacy, 
its image is the hypersurface determined by the equation  
\begin{equation}\label{eq:Equation_SABCD}
x^2+y^2+z^2+xyz=Ax+By+Cz+D
\end{equation}
with 
\begin{eqnarray}\label{eq:A_B_C}
A= ab+cd,\  B= bc+ad,\ C= ac+bd,\\ 
D=4 -(a^2+b^2+c^2+d^2)-abcd, \quad \label{eq:D}
\end{eqnarray}
and $\chi$ is the quotient map for the action of $\SL_2(\C)$ on $\Rep(\Sphere_4^2)$ by conjugacy, in the sense of invariant theory. 
We shall denote the character variety $\Rep(\Sphere_4^2)/\!\!/\SL_2(\C)$ by $\chi(\Sphere_4^2)$ (instead of $\chi(\Rep(\Sphere_4^2))$).

For each choice of parameters $A$, $B$, $C$, $D$, we shall denote by $S_{(A,B,C,D)}$ the algebraic surface determined
by the Equation~\eqref{eq:Equation_SABCD}; it is a cubic surface in the affine space $\A^3$, of degree $2$ with respect to each variable
$x$, $y$, or $z$. The family of all 
these surfaces will be denoted by $\Fam$. For simplicity, we shall just write $S$ instead of $S_{(A,B,C,D)}$ for the elements of $\Fam$ 
(the quadruple $(A,B,C,D)$ is then uniquely determined by $S$). If
the parameters $A$, $B$, $C$, and $D$ are in a ring $R$, we denote by $S(R)$ the points of $S$ with coordinates in $R$. 

The compactification of $S$ in the projective space $\P^3$ will be denoted by 
$\Sbar$ and its boundary at infinity by
\begin{equation}
\partial S  = \Sbar  \setminus S.
\end{equation} 
For every $S\in \Fam$,  $\partial S$ is the triangle in the hyperplane at infinity given by the equation $xyz=0$.

\subsection{Mapping class group and Vieta involutions}\label{par:vieta_involutions}  
Let us view the mapping class group $\Mod^{\pm}(\Sphere_4^2)$  as the subgroup of $\Out(\pi_1(\Sphere_4^2))$ preserving the peripheral structure. Then,  $\Mod^{\pm}(\Sphere_4^2)$ 
acts on $\chi(\Sphere_4^2)$ by precomposition, the conjugacy class of a representation $\rho$ being sent to the conjugacy class
of $\rho\circ \Phi^{-1}$ for any mapping class $\Phi$. This gives a homomorphism $\Phi\mapsto f_\Phi$ into the group $\Aut(\chi(\Sphere_4^2))$ of automorphisms of the variety $\chi(\Sphere_4^2)$. As explained in~\cite[Sections 2.2 and 2.3]{Cantat-Loray}, there is an isomorphism
\begin{equation}
\PGL_2(\Z)\ltimes H\simeq\Mod^{\pm}(\Sphere_4^2)
\end{equation}
where $H$ is the group $(\Z/2\Z)^2$, and an exact sequence 
\begin{equation}
\Id\to \Gamma^{\pm}_2\to \PGL_2(\Z)\ltimes H\to \Sym(4)\to \Id
\end{equation}
where $\Sym(4)$ corresponds to the group of permutation of the four punctures of $\Sphere_4^2$ and $\Gamma^{\pm}_2$ 
is the congruence subgroup of $\PGL_2(\Z)$ modulo $2$. 
\begin{enumerate}
\item  $\Gamma^{\pm}_2$ is isomorphic to $\Z/2\Z\star \Z/2\Z\star \Z/2\Z$, generated by the three involutions
\begin{equation}\label{eq:involution_matrices}
\hat\sigma_x=\left(\begin{array}{cc} -1 & 2 \\ 0 & 1\end{array}\right), \ 
\hat\sigma_y=\left(\begin{array}{cc} 1 & 0 \\ 2 & -1\end{array}\right) , \ 
\hat\sigma_z=\left(\begin{array}{cc} 1 & 0 \\ 0 & -1\end{array}\right);
\end{equation}
\item $\Gamma^{\pm}_2$ is the reflection group of an ideal triangle in the upper half plane; 
\item the action of $\Gamma^{\pm}_2$ on the character variety $\chi(\Sphere_4^2)$ preserves the four coordinates $a$, $b$, $c$, $d$; 
in particular, it acts on each surface $S\in \Fam$ as a group of automorphisms of $S$. This action is faithfull for every $S$, and the image
has index at most $24$ in $\Aut(S)$ (see~\cite[Theorem 3.1]{Cantat-Loray}). 
\item the image  of $\Gamma^{\pm}_2$ in $\Aut(S)$ is generated by the Vieta involutions 
\begin{align}\label{eq: vieta}
s_x(x,y,z)&=(-x-yz+A, y, z) \\ 
s_y(x,y,z)&=(x, -y-zx+B, z)  \\ 
s_z(x,y,z)&=(x, y, -z-xy+C).
\end{align}
\end{enumerate}

In what follows, we denote by $\Gamma$ the abstract group $\Z/2\Z\star \Z/2\Z\star \Z/2\Z$. Depending on the action we look at, 
$\Gamma$ will determine a subgroup $\Gamma^{\pm}_2$ of isometries of the upper half plane, or a group of automorphisms $\Gamma_S$ of $S$ (for any $S$ in $\Fam$). Our goal is to describe the stochastic dynamics of this action on each of the surfaces $S$. 
For simplicity, we frequently write $\Gamma$ instead of $\Gamma_S$. 

\subsection{Invariant area form}\label{par:invariant_forms}  
On $S$, the $2$-form 
\begin{equation}
\area=\frac{dx\wedge dy}{2z+xy-C}=\frac{dy\wedge dz}{2x+yz-A}=\frac{dz\wedge dx}{2y+zx-B}
\end{equation}
is regular and does not vanish. 
According to Lemma 3.5 of~\cite{Cantat-Loray}, singularities of $S$ are quotient singularities, and $\area$ is an area form in the sense of orbifolds (locally, in a euclidean neighborhood of the singularity, $S$ is the quotient $\C^2/G$ for some finite group and $\area$ is the quotient of a $G$-invariant  symplectic $2$-form on $\C^2$). 

When $S$ is defined over $\R$, $\area$ is also defined over $\R$. To distinguish between the holomorphic $2$-form on $S(\C)$ and the real $2$-form on $S(\R)$, we use the notation $\area_\C$ and $\area_\R$.

The $2$-form $\area_\C$ (resp. $\area_\R$) is multiplied by $-1$ under the action of each of the involutions $s_x$, $s_y$, $s_z$.

When $S$ is defined over $\R$, $S(\R)$ may have a compact connected component; we denote such a component by $S(\R)_c$ (see Section~\ref{par:topology_of_S} for a precise definition when $S$ is singular). With our notation, it may happen that $S(\R)_c$ be reduced to a point; otherwise, it is a (possibly singular) sphere and the restriction of $\area_\R$ to $S(\R)_c$ determines (a) an orientation of this sphere and (b) a probability measure $\nu_\R$ on $S(\R)_c$, defined by 
\begin{equation}
\nu_\R(B)=\frac{1}{\int_{S(\R)_c}\area_{\R}}\int_{B}\area_\R
\end{equation}
for any borel subset $B$ of $S(\R)_c$.
We shall refer to this measure as the {\bf{symplectic measure}} on $S(\R)_c$. It is $\Gamma_S$-invariant. 

\subsection{Random dynamics}\label{par:introduction_random_dynamics}
Let $X$ be a locally compact metric space. We endow the group $\Homeo(X)$ with the compact-open topology.
Let $\mu$ be a probability measure on  $\Homeo(X)$. 
Denote by $\Omega$ the product space $\Homeo(X)^\N$ and endow it with the probability measure $\mu^\N$.
We shall say that a property holds for a {\bf{typical}} element $\omega\in \Omega$ if it holds for $\omega$ in a measurable subset $\Omega'$ with $\mu^\N(\Omega')=1$.
For every $\omega = (f_0,f_1,\ldots) \in \Omega$, we set 
\begin{equation}
f_\omega^n = f_{n-1} \circ \ldots \circ f_0.
\end{equation} 
A probability measure $\nu$ on $X$ is $\mu$-stationary if 
\begin{equation}\label{eq:definition_stationary}
\int_{\Homeo(X)}(f_*\nu) \; \dd\mu(f)=\nu,
\end{equation}
and such a measure is ergodic if it can not be written as a convex combination of two distinct stationary measures. 
Fix a point $q$ in $X$. Then, consider the random orbit $(f_\omega^n(q))$ and the empirical measures 
\begin{equation}\label{eq:empirical_measures}
\nu_N(\omega; q)=\frac{1}{N}\sum_{j=1}^N\delta_{f_\omega^j(q)}.
\end{equation}
A theorem of Breiman (see~\cite{Benoist-Quint:Book}) says that, for a subset of measure
$1$ in $\Omega$, if $\nu_{n_i}(\omega; q)$ converges towards a probability measure $\nu$ as $n_i$ goes to infinity, then $\nu$ is $\mu$-stationary. Thus, to understand how random orbits  distribute, we have to describe
stationary measures. 
We apply this viewpoint to the dynamics of~$\Gamma_S$. 

\vspace{0.2cm}

\begin{thm-M}--  Let $\mu$ be a probability measure on $\{s_x,s_y,s_z\}$ with $$\mu(s_x)\mu(s_y)\mu(s_z)>0.$$ 
Let $S$ be an element of $\Fam$. Let $\nu$ be a probability measure on $S(\C)$. If $\nu$ is $\mu$-stationary and ergodic, 
then the support of $\nu$ is compact, $\nu$ is invariant, and either $\nu$ is given by the average on a finite orbit of $\Gamma_S$, or
\begin{itemize}
\item   the parameters $A$, $B$, $C$, and $D$ defining $S$ are real and in $[-2,2]$,
\item  $S(\R)$ has a unique bounded component $S(\R)_c$, of dimension $2$, and
\item  $\nu$ coincides with the symplectic measure $\nu_\R$ induced by $\area_\R$ on $S(\R)_c$.
\end{itemize}
\end{thm-M}

\vspace{0.2cm}

This extends a result of Chung for the dynamics on the compact part $S(\R)_c$ when $(A,B,C,D)=(0,0,0,D)$ with $D\in [3.9,4[$ (see~\cite[Theorem B]{Chung}).

Since finite orbits have been classified, our Main Theorem gives a complete description of all $\mu$-stationary measures; in particular, 
{\sl{on each surface $S\in \Fam$, except on $S_{(0,0,0,4)}$,  the set of $\mu$-stationary measures is a finite dimensional simplex with at most $6$
vertices, the maximum $6$ being realized only by $S_{(0,0,0,3)}$}}. 

Our result certainly holds when the support of $\mu$ generates the group
$\Gamma$ and
satisfies an exponential moment condition. We wrote the proof assuming that the
support is equal to
$ \left\{ s_x, s_y, s_z  \right\}$
to simplify the exposition, as it
allows us to avoid the use of the general theory of random walks on $\Gamma$.


\subsection{Acknowledgement} We thank Romain Dujardin, Sébastien Gouëzel, Vincent Guirardel, Seung uk Jang, Frank Loray, Arnaud Nerrière and Juan Souto for interesting discussions related to this paper.

\section{The Markov surfaces and their automorphisms}

\subsection{Smoothness of $S(\C)$}\label{par:smoothness}
In~\cite{Benedetto-Goldman}, Benedetto and Goldman study the topology of the surfaces $S\in \Fam$. Their first result says that $S$ is singular if, and only if at least one of the parameters $a$, $b$, $c$, or $d$ 
is equal to $\pm 2$ or there 
is a reducible representation with boundary traces $a$, $b$, $c$, $d$; moreover, the latter case occurs if and only if the following discriminant vanishes:
\begin{equation*}
\Delta=(2(a^2+b^2+c^2+d^2)-abcd-16)^2- (4-a^2)(4-b^2)(4-c^2)(4-d^2).
\end{equation*}

\begin{eg}\label{eg:cayley_cubic_I}
The {\bf{Cayley cubic}} $S_{\mathrm{Ca}}=S_{(0,0,0,4)}$ is  defined by the equation $x^2+y^2+z^2+xyz=4$. It has four singularities, the maximum for an irreducible cubic surface. It is the quotient of the multiplicative group ${\mathbb{G}}_m\times{\mathbb{G}}_m$ by $\eta(u,v)=(1/u,1/v)$, the quotient map being 
\begin{equation}
(u,v)\mapsto (-u-1/u, -v-1/v, -uv-1/(uv)).
\end{equation}
\end{eg}

\subsection{Topology of $S(\R)$}\label{par:topology_of_S} 
Let $n$ be the number of boundary traces $a$, $b$, $c$, $d$ in the interval $]-2,2[$. According to~\cite[Theorem 1.2]{Benedetto-Goldman}, if $S(\R)$ is smooth then its Euler characteristic is equal to $2n-2$ and $S(\R)$ is homeomorphic to 
\begin{enumerate}
\item a quadruply punctured sphere if $n=0$ and $abcd<0$;
\item a disjoint union of a triply punctured torus and a disk if $n=0$ and $abcd>0$;
\item a disjoint union of a triply punctured sphere and a disk if $n=1$;
\item a disjoint union of an annulus and two disks if $n=2$;
\item a disjoint union of four disks if $n=3$;
\item a disjoint union of four disks and a sphere if $n=4$.
\end{enumerate}
In particular, if $S(\R)$ has a compact connected component, then this component is unique and is homeomorphic to a sphere. 

When $S$ is allowed to have singularities, we define the {\bf{components}} of $S(\R)$ as follows. Firstly, if $S(\R)$ has an isolated point, this point will be one of the components of $S(\R)$; note that there is at most one isolated point, it must be one of the singularities of $S(\R)$, and a well chosen perturbation of  the coefficients $(a, b, c, d)$ will turn this point into a small sphere (resp. will make this point disappear). 
Secondly,  consider the smooth part $S(\R)\setminus \Sing(S)$, and split it as a disjoint union of connected components $S(\R)_i^0$; then, replace each $S_i^0(\R)$ by its closure (equivalently, add to $S(\R)_i^0$ the singularities of $S$ which are contained in the closure of $S(\R)_i^0$); the result 
\begin{equation}
S(\R)_i:=\overline{S(\R)_i^0}
\end{equation}
will be one of the components of $S(\R)$. For instance, in the case of the Cayley cubic, $S(\R)$ has a compact component: it contains four singularities, each of which is contained in exactly one of the unbounded components. 

With this definition at hand, the classification of Benedetto and Goldman remains correct, except that the components are not necessarily disjoint (they may touch at singularities) and the sphere in Case~(6) can be reduced to a singular point. In particular, one gets the following result: {\sl{if $S(\R)$ has a compact component, this component is unique, and it is homeomorphic to a sphere or a singleton}}. If such a component exists, it will be denoted by $S(\R)_c$; when $S(\R)$ has an isolated point, then $S(\R)_c$ is equal to this point. 

\subsection{Representations in the compact component}
The group $\SL_2(\C)$ has two real forms, one is $\SL_2(\R)$, the other is $\SU_2$. Since $\SU_2$ is compact, representations in $\SU_2$ give points in 
$S(\R)_c$; but points of $S(\R)_c$ may also correspond to representations in $\SL_2(\R)$ (see~\cite[Proposition 1.4]{Benedetto-Goldman}). In fact, according to~\cite[Lemma 2.7]{Cantat-Loray}, the map
\begin{equation}
\Pi\colon (a,b,c,d)\in \C^4\mapsto (A,B,C,D)\in \C^4
\end{equation} 
from Equations~\eqref{eq:A_B_C} and~\eqref{eq:D}
is a ramified cover of degree $24$ and ${\mathrm{Jac}}(\Pi)=-\frac{1}{2}\Delta$, where $\Delta$ is the discriminant from Section~\ref{par:smoothness}. 
The following automorphisms of $\C^4$ generate a group $Q\subset \Aut(\C^4)$ of order $8$:
\begin{itemize}
\item[(a)] the simultaneous sign change of the parameters $a$, $b$, $c$, and $d$
\item[(b)] the  permutations of $a$, $b$, $c$, and $d$
 which are a composition of two transpositions with disjoint supports.
\end{itemize}
The ramified cover $\Pi$ is invariant under the action of $Q$. But $\Pi$ is not Galois, and to understand the structure of the $24$ points in the fibers of $\Pi$, one has to introduce the Okamoto correspondences: see Section 3 of~\cite{Cantat-Loray} for this.

 If $S(\R)$ has a compact component, then $(A,B,C,D)$ is contained in $[-2,2]^4$ and $\Pi^{-1}\{(A,B,C,D)\}$ is also entirely contained in $[-2,2]^4$. Moreover, different choices of $(a,b,c,d)$ in 
$\Pi^{-1}\{(A,B,C,D)\}$ lead to different types of representations, as summarized in the following result (see~\cite[Theorem B and Proposition 3.13]{Cantat-Loray}).
\begin{thm}
Let $(A,B,C,D)$ be real parameters. If the smooth part of $S(\R)$ has a bounded component, then all parameters $(a,b,c,d)$ in $\Pi^{-1}\{(A,B,C,D)\}$ are real. 
Moreover, the points of $S(\R)_c$ are conjugacy classes of  
 $\SU_2$ and $\SL_2(\R)$-representations, depending on the choice of $(a,b,c,d)$: modulo the action of the group $Q$, each point of $S(\R)_c$ corresponds to two $\SU_2$-representa\-tions and one $\SL_2(\R)$-representation. \end{thm}

\section{Invariant compact subsets and invariant measures on them}

In this section, we summarize some known results concerning $\Gamma_S$-invariant compact subsets of $S(\C)$ and derive from them a classification of $\Gamma_S$-invariant probability measures on $S(\C)$.

\subsection{Finite orbits}
Boalch, and Lisovyy and Tykhyy obtained a classification of all finite orbits of $\Gamma$ in the surfaces $S\in \Fam$. We summarize their results.

\subsubsection{Short orbits}\label{par:short_orbits} 
Finite orbits of $\Gamma$ with at most $4$ elements are classified in~\cite[Lemma 39]{Lisovyy-Tykhyy}.
These finite orbits come in families, depending on $3$, $2$, or $1$ parameters, up to permutation of the coordinates. Fixed points coincide with singularities of $S$, and form a $3$-parameter family, the parameters $(a,b,c,d)$ being on the locus $(4-a^2)(4-b^2)(4-c^2)(4-d^2)\Delta=0$. Orbits of length $2$ depend on $2$ parameters, for instance $\{(x,0,0), (x',0,0)\}$ is such an orbit if $A=x+x'$, $B=0$, $C=0$, and $D=4+x+x'$. Orbits of length $3$ or $4$ depend on $1$ parameter; for a generic choice of the parameter, the points in the orbit correspond to representations of $\pi_1(\Sphere_4^2)$ with an infinite image. 

\subsubsection{Finite groups}\label{par:finite_groups_orbits} Let $F$ be a finite subgroup of $\SL_2(\C)$. Then, each representation of $\pi_1(\Sphere_4^2)$ into $F$ gives rise to a finite orbit of $\Gamma$ 
in the character variety $\chi(\Sphere_4^2)$. 
Changing $F$ into a conjugate subgroup of $\SL_2(\C)$ does not change the corresponding orbit. 
The finite orbits obtained with this method have been described by Boalch in~\cite{Boalch:Fourier2003, Boalch:PLMS2005, Boach:Crelle2006}. Note that changing the parameters $(a,b,c,d)$ in $\Pi^{-1}(A,B,C,D)$ may turn a representation inside a finite group into a representation with infinite image. 

\begin{eg} Consider the Cayley cubic $S_{\mathrm{Ca}}=S_{(0,0,0,4)}$ described in Example~\ref{eg:cayley_cubic_I}. The group $\GL_2(\Z)$ acts by 
monomial transformations on the multiplicative group ${\mathbb{G}}_m(\C)\times{\mathbb{G}}_m(\C)$ and each point of type $(e^{2\ii\pi p/q}, e^{2\ii\pi p'/q'})$ has a finite orbit for this action. 
On the other hand, this action commutes to the involution $\eta$ and induces a subgroup of $\Aut(S)$ that contains 
$\Gamma_{S_{\mathrm{Ca}}}$ as a finite index subgroup (see~\cite{Cantat-Loray}). Thus, their projections 
\begin{equation}
(-2\cos(2\ii\pi p/q), -2\cos(2\ii\pi p'/q'), -2\cos(2\ii\pi (p/q+p'/q')))
\end{equation} 
give points on the Cayley cubic with finite $\Gamma_{S_{\mathrm{Ca}}}$-orbits. The Cayley cubic is the unique example of a surface $S\in \Fam$
containing infinitely many finite orbits.
\end{eg}

\subsubsection{Boalch-Klein orbit}\label{par:boalch-klein} In~\cite{Boalch:PLMS2005}, Boalch constructs a finite orbit that is not given by any representation into a finite subgroup of $\SL_2(\C)$ (it comes from a representation into a finite subgroup of $\SL_3(\C)$, though). 
The surface is defined by the equation
\begin{equation}
x^2+y^2+z^2+xyz=x+y+z
\end{equation}
and the orbit is made of the seven points
$(0,0,0)$, $(1,0,0)$, $(0,1,0)$, $(0,0,1)$,
$(1,1,0)$, $(1,0,1)$, and $(0, 1,1)$. 
There are $24$
choices of parameters $(a,b,c,d)$ giving rise to $(A,B,C,D)=(1,1,1,0)$.
One of them is 
\begin{equation}
a=b=c=2\cos(2\pi/7), \quad d= 2\cos(4\pi/7).
\end{equation}
For such a choice, the image of the representation $\pi_1(\Sphere_4^2)\to \PSL_2(\C)$ is conjugate to the triangle group $T(2,3,7)$.
The Galois conjugates of this representations 
provide two distinct representations in $\SU_2$ (corresponding to conjugates of $a$, $b$, $c$, $d$). Modulo Okamoto symmetries, 
all orbits of length $7$ are given by these three representations.

\subsubsection{Classification} 
Lisovyy and Tykhyy proved that {\sl{every finite orbit of $\Gamma$ is given by one of the above examples}}. 
The parameters for finite orbits are listed in Lemma~39 page 147, Theorem 1 page 149 and Table 4 page 150 of~\cite{Lisovyy-Tykhyy}. Each point with a finite orbit of length $\geq 5$ is determined by a representation of $\pi_1(\Sphere_4^2)$ in a finite subgroup of $\SU_2$, except for the Boalch-Klein orbit. 

\subsection{Invariant compact subsets and probability measures}
From Lemma 4.3 and Theorems~B and~C of~\cite{Cantat-Loray}, one gets the following result.
\begin{thm}\label{thm:compact_invariant} Let $S$ be an element of $\Fam$.  If $\Gamma_S$ preserves a compact subset $K$ of $S(\C)$ with at least five elements then 
\begin{enumerate}
\item the parameters $a$, $b$, $c$, and $d$ are in $[-2,2]$;
\item $S(\R)$ has a unique compact component $S(\R)_c$;
\item  $K$ is   either finite or   equal to $S(\R)_c$.
\end{enumerate}
\end{thm}

In~\cite{Goldman:Geometry_Topology} and~\cite{Pickrell-Xia:II},  Goldman, Pickrell and Xia prove 
that the symplectic measure $\nu_\R$ is ergodic for the action of $\Gamma_S$ on $S(\R)_c$. 
Together with Theorem~\ref{thm:compact_invariant}, this gives a classification of invariant probability measures with compact support
(and, as we shall see in Corollary~\ref{cor:classification_of_invariant_measures_II}, a classification of all invariant probability measures): 

\begin{cor}\label{cor:classification_of_invariant_measures_I} Let $\nu$ be a $\Gamma_S$-invariant  probability measure   on $S(\C)$ with compact support. If 
$\nu$ is ergodic, then   
 either $\nu$ is the counting measure on a finite orbit of $\Gamma_S$, 
  or the parameters $a$, $b$, $c$, and $d$ are in $[-2,2]$, $S(\R)$ has a unique compact component and $\nu$ is the symplectic measure $\nu_\R$ introduced in Section~\ref{par:invariant_forms}.
 \end{cor}
 
 Since the proof is a simple variation on the arguments of~\cite{Goldman:Geometry_Topology}, \cite{Cantat-Loray}, and~\cite{Cantat-Dujardin:Transformation-Groups}, we only sketch it. 
 
\begin{proof}
Denote by $K$ the support of $\nu$. 
If $\nu$ has an atom at a point $q\in S(\C)$ then the orbit of $q$ contains at most $\nu(\{q\})^{-1}$ points 
because $\nu$ is a probability measure. Then, by ergodicity,  $\nu$ coincides with the counting measure on $\Gamma_S(q)$.

Now, suppose that $\nu$ does not have any atom. If there is an irreducible algebraic curve $C\subset S(\C)$ with $\nu(C)>0$, then 
$\nu(f(C)\cap C)=0$ for every $f\in \Gamma_S$ that does not preserve $C$. Thus, the orbit of $C$ under $\Gamma_S$ is a  union of at most $\nu(C)^{-1}$  irreducible curves and we obtain a contradiction because $\Gamma_S$ does not preserve any curve (see~\cite[Theorem~D]{Cantat-Loray}). Thus, we may assume that $K$ is infinite and $\nu(Z)=0$ for every proper algebraic subset of $S(\C)$. 

From Theorem~\ref{thm:compact_invariant}, the parameters $A$, $B$, $C$, and $D$ are real and 
$K$ coincides with $S(\R)_c$. Consider the projection $\pi\colon S\to \R$ onto the first axis; its image is contained in $[-2,2]$. The composition of the second and third involutions acts on $S$ by 
\begin{equation}
g(x,y,z)= (x, -xz-y+B, x^2z-z+C-Bx).
\end{equation}
If the fiber $S(\R)_{x_0}:=\pi^{-1}\{x_0\}$ is non-empty and smooth, then it is an ellipse, it is diffeomorphic to a circle, and $g$ acts on it as a rotation, the angle of which is given by $2cos(\pi \theta)=x_0$ (see~\cite[\S 5]{Goldman:Annals} and~\cite{Cantat-Loray}).
Thus, if $x_0$ is not in $2\cos(2\pi \Q)$, $g$ preserves a unique probability measure $\lambda_{x_0}$ on $S(\R)_{x_0}$: this measure is smooth, it corresponds to the Lebesgue measure on the circle. 

Consider the projection $\nu_x:=\pi_*(\nu)$. This measure does not have any atom. Thus, if we desintegrate $\nu$ with respect to the fibration $\pi$, the conditional measures coincide with the measures $\lambda_{x_0}$ because they are $g$-invariant. Now, projecting on the second axis and applying the same argument, we obtain a smooth measure $\nu_{y}$ on the base, with smooth conditional measure. This shows that $\nu$ is absolutely continuous with respect to the symplectic measure $\nu_\R$, and by the  ergodicity theorem of~\cite{Goldman:Geometry_Topology, Pickrell-Xia:II}, $\nu$ is equal to $\nu_\R$.
\end{proof}


\section{Dynamics at infinity}\label{par:Dynamics at infinity}

We extend   $\Gamma_S$ as a group of birational transformations of  $\Sbar$, the completion of $S$ in $\P^3$, and study its dynamics near the hyperplane at infinity. 

\begin{rem}\label{rem:local_fields_1}
We present all computations and estimates for the complex surface $\SC$. On the other hand, the same computations work equally well on $S(\mathbb{K})$ if $\mathbb{K}$ is a local field, for some non-trivial absolute value $\vert \cdot\vert$. In that case, $\parallel \cdot \parallel$ must be  a norm on ${\mathbb{K}}^3$ which is compatible with $\vert \cdot \vert$. Good examples to keep in mind are ${\mathbb{K}}=\Q_p$, or ${\mathbb{K}}={\mathbb{F}}_p[t]$, $p$ a prime.
\end{rem}

\subsection{Taylor expansion at infinity}\label{ss: action si}
  We add a variable $w$ to get homogeneous coordinates $[x:y:z:w]$ on $\P^3$; then, the hyperplane at infinity is $\{w=0\}$, and its intersection with $\Sbar$  is the triangle $\partial S=\Sbar\setminus S$ given by the equation $xyz=0$. The vertices of this triangle are $p_1=[1:0:0:0]$, $p_2=[0:1:0:0]$, and $p_3=[0:0:1:0]$. Near $p_3$, we can use the local coordinates $x,y,w$, with $z=1$. The equation of $S$ becomes, locally, 
\begin{equation}
(1+x^2+y^2)w+xy=(Ax+By+C)w^2 +Dw^3.
\end{equation}
Thus, $S$ is smooth near $p_3$, its tangent plane at $p_3$ is the ``horizontal'' plane $w=0$, and $S$ is locally the graph of a 
function $\varphi_3\colon (x,y)\mapsto w=\varphi_3(x,y)$. Since $\Sbar$ intersects the plane $\{w=0\}$ on the coordinate axis, $\varphi_3$ is 
divisible by $xy$; then, its Taylor expansion starts by
\begin{align*}
\frac{\varphi_3(x,y)}{-xy}&=1 -( x^2+Cxy+y^2) -(Ax+By)xy  \\
\notag& \quad \quad  +(x^4+3Cx^3y+(2-D+2C^2)x^2y^2+3Cxy^3+y^4)+\ldots
\end{align*}
Hence 
\begin{equation}\label{eq: phi3}
\frac{\varphi_3(x,y)}{-xy} =1 -( x^2+Cxy+y^2) -(Ax+By)xy + O(\parallel(x,y)\parallel^4),
\end{equation}
where the norm is   any (fixed) norm in the local $(x,y)$-coordinates. The coefficients in this expansion are polynomial functions of $(A,B,C,D)$ with integer coefficients. 

Similarly, we can write locally $\Sbar$ as a graph $(y,z)\mapsto w=\varphi_1(y,z)$ near $p_1$, and as a graph $(z,x)\mapsto w=\varphi_2(z,x)$ near $p_2$ (note the cyclic permutation of the local coordinates); we shall use the notation 
\begin{itemize}
\item $(u_1,v_1, \varphi_1(u_1,v_1))=(y,z,\varphi_1(y,z))$ near the point $p_1$,
\item $(u_2,v_2, \varphi_2(u_2,v_2))=(z,x,\varphi_2(z,x))$ near the point $p_2$, and
\item $(u_3,v_3, \varphi_3(u_3,v_3))=(x,y, \varphi_3(x,y))$ near the point $p_3$.
\end{itemize}
More precisely, $(u_i,v_i, w)$ are local coordinates near $p_i$, the projection from $\Sbar$ to the $(u_i,v_i)$-plane being a local diffeomorphism. 

We can now compute the Taylor expansion of the involutions $s_x$, $s_y$, and $s_z$ near the points $p_i$ where they are well defined. 
For instance, at $p_3$ both $s_x$ and $s_y$ are well defined, while $s_z$ has an indeterminacy. From Equation~\eqref{eq: vieta}, \begin{equation}\label{eq: sxp}
 s_x[x:y:z:w] = [-xw-yz+A w^2 : yw : zw : w^2] 
 \end{equation}
Since $s_x$ maps $p_3 = [0:0:1:0]$ to $p_1 = [1:0:0:0]$, we use the local coordinates $(u_3,v_3)=(x,y)$ and 
$(u_1,v_1)=(y,z)$. 
Combining Equations~\eqref{eq: phi3} and~\eqref{eq: sxp} with $w=\varphi_3(x,y)$, we can write $s_x(x,y)=(y',z')$ with  
\begin{align}
y' &= xy[1-Cxy-y^2-Bxy^2]+\cdots\\
z' &= x[1-Cxy-y^2 -Bxy^2-x^2(x^2+2Cxy+2y^2)]+\cdots
 \end{align}
 where the dots correspond to terms of degree at least $6$.
Thus if we start with the point of coordinates $(u_3,v_3)$, that is with the point $[u_3:v_3:1: \varphi_3(u_3,v_3)]\in S$, the coordinates $(u_1',v_1')$ of its image by $s_x$ satisfy
\begin{align}
u_1' &= u_3v_3[1+O(\parallel (u_3,v_3)\parallel^2)]\\
v_1'&= u_3[1+O(\parallel (u_3,v_3)\parallel^2)]. \end{align}
This is not a surprise, indeed locally $s_x$ is the blow down of the axis $\{u_3=0\}$. 
A similar computation can be done near each of the vertices of the triangle $\partial S$. To summarize it, we introduce the following two matrices. 
\begin{equation}
A := \begin{pmatrix} 0 & 1 \\ 1 & 1 \end{pmatrix} \ , \  B := \begin{pmatrix} 1 & 1 \\ 1 & 0 \end{pmatrix} 
\end{equation}
We denote by $(u,v)\mapsto (u,v)^M$ the monomial action of a matrix $M\in \Mat_2(\Z)$; for instance, $(u,v)^A=(v,uv)$ and $(u,v) ^B=(uv, u)$.

\begin{figure}[h] 
\resizebox{.98\textwidth}{!}{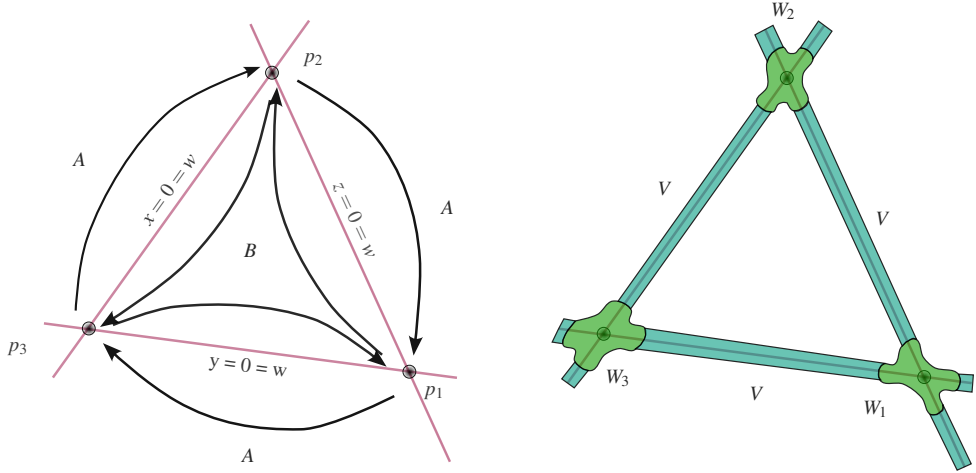}
\caption{{\small{ 
{\bf{On the left}}, one sees the triangle at infinity $\partial S$ with the three vertices corresponding to the indeterminacy points of the involutions. For instance, $s_x$ maps the line $\{w=0=x\}$ to the point $[1:0:0:0]$, and locally near $[0:1:0:0]$ it behaves like the monomial map $(u_2,v_2)^A$. {\bf{On the right}}, one sees the two types of open sets used to confine the dynamics near infinity: small neighborhoods $W_i$ near the vertices, in green, small neighborhoods $V$ around the edges; the neighborhoods of type $V$ depend on the words. }}}
 \label{fig:triangle} 
\end{figure}

\begin{pro}\label{pro:monomial_at_infinity}
In the local coordinates $(u_1,v_1)$ near the point $p_1\in \Sbar$, $(u_2,v_2)$ near $p_2$, and $(u_3,v_3)$ near $p_3$, the involution $s_x$ acts by
\begin{align*}
s_x(u_2,v_2)&=(v_2(1+O(\parallel (u_2,v_2)\parallel^2)), u_2v_2(1+O(\parallel (u_2,v_2)\parallel^2))\\
s_x(u_3,v_3)&=(u_3v_3(1+O(\parallel (u_3,v_3)\parallel^2),v_3(1+O(\parallel (u_3,v_3)\parallel^2)).
\end{align*}
Thus, $s_x$ acts as the monomial map $(u_2,v_2)^A$ near $p_2$ and as $(u_3,v_3)^B$ near $p_3$, up to multiplication by functions  of 
type $1+O(\parallel (u_i,v_i)\parallel^2)$. 

Similar formulas hold for $s_y$ and $s_z$, where the choice of $A$ or $B$ in the monomial
expansion is given on Figure~\ref{fig:triangle}. \end{pro}

Note that if we set $\alpha_i=-\log\vert u_i\vert$ and $\beta_i=-\log\vert v_i\vert$, then the action of $s_x$ near $p_2$ transforms $(\alpha_2,\beta_2)$
in $(\alpha_1',\beta_1')$ with 
\begin{align}
\label{eq:log_monomial1}\alpha_1'&=\beta_2+O(\exp(-2\parallel (\alpha_2,\beta_2)\parallel))\\
\label{eq:log_monomial2}\beta_1'&=\alpha_2+\beta_2+ O(\exp(-2\parallel (\alpha_2,\beta_2)\parallel)).
\end{align}
Here, as $(u_2,v_2)$ approaches $\partial S$ in $S$, then $(\alpha_2,\beta_2)$ goes to $\infty$ in $\R^2_+$.

\subsection{Dynamics of the semi-group $\langle A,B \rangle_{s.g.}$ on $\R_+^2$}\label{ss:linear_dynamics_A_B}
The last paragraph shows that, to describe the action of $\Gamma_S$ near $\partial S$, we must look at the linear dynamics of the semi-group generated by $A$ and $B$ on $\R^2_+$. This semi-group is the free semi-group on two generators $\langle A, B\rangle_{s.g.}$; its elements are uniquely represented by words in $A$, and $B$, such as $w(A,B)=A A  B A  B B A A A B$ (we do not distinguish a word in $A$ and $B$ from the corresponding element in $\langle A, B\rangle_{s.g.}$); the length of such a word is just its number of letters; a sequence of words $(w_n)$ is said to be increasing if $w_{n+1}=v_n \cdot w_n$ is the concatenation of $w_n$ with a word $v_n$ of positive length. In the next lemma, $\parallel \cdot\parallel_1$ is the $\ell_1$-norm on $\R^2$. 

\begin{lem}\label{lem:first_linear_estimate}
Let $(\alpha, \beta)$ be a point of $\R^2_+$. Let $(w_n)$ be an increasing sequence of words in the semi-group $\langle A,B \rangle_{s.g.}$. 
Then, given any point $(\alpha, \beta)\neq (0,0)$ in $\R^2_+$, 
\begin{itemize}
\item either $\parallel w_n(\alpha,\beta)\parallel_1$ goes to $+\infty$ as $n$ goes to $+\infty$,
\item or $\alpha=0$ and $w_n= B^{\epsilon(n)}(AB)^{\ell(n)}$ 
for some increasing sequence $\ell(n)\in \Z_+$ and some sequence $\epsilon(n)\in\{0,1\}$,
\item or $\beta=0$ and $w_n=A^{\epsilon(n)}(BA)^{\ell(n)}$ 
for some increasing sequence $\ell(n)\in \Z_+$ and some sequence $\epsilon(n)\in\{0,1\}$.
\end{itemize}
If $\alpha\beta>0$, the orbit of $(w_n(\alpha, \beta))$ of $(\alpha,\beta)$ in $\R^2_+$ is discrete and contains at most $\min(\alpha,\beta)^{-1}R$ points at distance $\leq R$ from the origin.
\end{lem} 

\begin{proof} For any $(\alpha, \beta)\in \R^2_+$ and for $U\in \{A,B\}$ we have   $\parallel U(\alpha, \beta)\parallel_1\geq \parallel (\alpha,\beta)\parallel_1+\min(\alpha,\beta)$. Thus, if $w_m(\alpha,\beta)$ is in the interior of $\R^2_+$ for some $m\geq 1$, then $\parallel w_n(\alpha,\beta)\parallel_1\geq a (n-m)$ for some $a>0$ and all $n$. Now, $w_n(\alpha,\beta)$ stays permanently on the boundary of $\R^2_+$ if and only if $\alpha=0$ (resp. $\beta=0$) and  $w_n$ is an alternating sequence of $A$ and $B$, as stated in the lemma. 
\end{proof}


\begin{lem}\label{lem:second_linear_estimate}
Let $(\alpha, \beta)$ be a point of $\R^2_+$. Let $(w_n)$ be an increasing sequence of words in the semi-group $\langle A,B \rangle_{s.g.}$. 
Let $(\alpha, \beta)$ be a point in $\R^2_+$ such that $\alpha\beta\neq 0$. Denote by $(\alpha_n,\beta_n)$ the coordinates of $w_n(\alpha,\beta)$.
\begin{itemize}
\item either $\min(\alpha_n,\beta_n)$ goes to $+\infty$ as $n$ goes to $+\infty$,
\item or there is an index $n_0$ such that, for all $n\geq n_0$, 
$$\min(\alpha_{n},\beta_n)=\alpha_{n_0}\;\; {\text{and}} \;\;w_n= B^{\epsilon(n)}(AB)^{\ell(n)}w_{n_0}$$
for an increasing sequence $\ell(n)\in \Z_+$ and some sequence $\epsilon(n)\in\{0,1\}$.
\end{itemize}
\end{lem}

The proof is the same as for the previous lemma. The reason why there is only one exceptional case is that we can change 
a sequence of type $(BA)^{\ell(n)}$ into a sequence of type $B(AB)^{\ell(n)-1}v$ with $v=A$ and concatenate $v$ with $w_{n_0}$.

Now, we consider the following process, acting again on $\R^2$.  
An increasing sequence of words $w_n=U_{L(n)}U_{L(n-1)}\cdots U_2U_1$ is given, but now the $U_k$ are not in $\{A,B\}$; each $U_k$ is a small (non-linear) perturbation of $A$ or $B$, of type 
\begin{equation}
U_k(\alpha,\beta)=V_k(\alpha,\beta)+P_k(\alpha,\beta)
\end{equation} 
where $V_k\in\{A,B\}$ and $\parallel P_k(\alpha,\beta)\parallel_1\leq C\exp(-2\parallel (\alpha,\beta)\parallel_1)$ 
for some fixed constant $C$ (that does not depend on $(\alpha, \beta)$ or $(w_n)$).

\begin{lem}\label{lem:perturbed_linear_dynamics} Let $R$ be a positive number such that $R\geq 2C\exp(-2R)$. 
 If $(\alpha,\beta)$ satisfies $\alpha\geq R$ and $\beta\geq R$, then for any sequence of words as above, we have
 $$ \parallel w_n(\alpha,\beta)\parallel_1\geq \parallel (\alpha,\beta)\parallel_1 + \frac{R}{2}n.$$
 In particular, $w_n(\alpha,\beta)$ goes to $\infty$ with $n$. Moreover, each coordinate of $w_n(\alpha,\beta)$ goes 
 to $+\infty$, except if there is an integer $n_0$ such that, for all $n\geq n_0$,
 $$w_n = B^{\epsilon(n)}(AB)^{\ell(n)}w_{n_0}$$ 
 for some increasing sequence $\ell(n)\in \Z_+$, some sequence $\epsilon(n)\in\{0,1\}$.
\end{lem}
Indeed, for $V\in \{A,B\}$ and any map $P\colon \R^2\to \R^2$ that satisfies $\parallel P(\alpha, \beta)\parallel_1\leq C\exp(-2\parallel (\alpha,\beta)\parallel_1)$, 
we obtain 
\begin{align}
\parallel V(\alpha,\beta)+ P(\alpha, \beta)\parallel_1&\geq \parallel (\alpha,\beta)\parallel_1+R-C\exp(- 2\parallel (\alpha,\beta)\parallel_1)\\
&\geq  \parallel (\alpha,\beta)\parallel_1+\frac{R}{2}.
\end{align} 
The proof is now the same as for Lemma~\ref{lem:first_linear_estimate} and Lemma~\ref{lem:second_linear_estimate}.

\subsection{Application: dynamics at infinity} Sections~\ref{ss:linear_dynamics_A_B} and~\ref{ss: action si} can be combined as follows. Fix some open neighborhoods $W_1$, $W_2$, $W_3$ of $p_1$, $p_2$, $p_3$ in $S$ with local coordinates $(u_i,v_i)$ as in~\S~\ref{ss: action si} (neighborhoods for which Proposition~\ref{pro:monomial_at_infinity} holds). 
Set $(\alpha_i,\beta_i)=(-\log\vert u_i\vert, -\log\vert v_i\vert)$; there is a constant $R_0$ such that $\{(u_i,v_i)\, ; \; \alpha_i\geq R_0 \; {\text{and}} \; \beta_i\geq R_0\}$ is contained in $W_i$ for each index~$i$. Then, Equations~\eqref{eq:log_monomial1} and~\eqref{eq:log_monomial2}  and their siblings for other choices of involutions and local coordinates say the following. There is a constant $C$ such that if $(u_2,v_2)$ is in $W_2$, then   $s_x(u_2,v_2)$ is in $W_1$, and if we write its coordinates in $W_1$ as $(u'_1,v'_1)=s_x(u_2,v_2)$ then their logarithms $\alpha'_1=-\log\vert u_1' \vert$ and $\beta'_1=-\log\vert v_1' \vert$ satisfy  
\begin{equation}
(\alpha_1', \beta_1')=A(\alpha_2,\beta_2)+P_{1,2}(\alpha_2,\beta_2)
\end{equation}
 where $P_{1,2}$ is defined on  $[R_0,+\infty[ \times [R_0,+\infty[$ and
\begin{equation}
\parallel P_{1,2}(\alpha,\beta)\parallel_1 \leq C \exp(-2 \parallel (\alpha,\beta) \parallel_1).
\end{equation} 
 Similar formulas hold for the other involutions in the open sets $W_i$ where they are well defined; the matrices $A$ or $B$ are to be chosen as on Figure~\ref{fig:triangle}, and the perturbations $P_{i,j}$ depend on the open sets;  the constant $C$ can be chosen uniformly, {i.e.} independently of the indices.  

We say that a word $w$ in $s_x$, $s_y$, $s_z$ is {\bf{reduced}} if $s_{i_{k+1}}\neq s_{i_k}$ for  all pairs of successive letters  in $w$.
A reduced word $w_n=s_{i_{\ell(n)}}\cdots s_{i_1}$ gives rise to a sequence $s_{i_1}$, $s_{i_2}\circ s_{i_1}$, $\cdots$, $s_{i_{\ell(n)}} \circ \cdots \circ s_{i_1}$ of birational transformations of  ${\overline{S}}$, for any $S \in \Fam$, such that $s_{i_k}\circ \cdots \circ s_{i_1}$ contracts the generic point of the boundary $\partial S$ onto a vertex $p_{i_k}$ that does not coincide with the indeterminacy point of~$s_{i_{k+1}}$. We use the same notation $w_n$ for the reduced word, the element of $\Gamma$ determined by it, and the associated automorphism of $S$, for any $S$ in $\Fam$. On the other hand, we prefer to use $ s\circ s'$ to denote the composition in $\Gamma_S\subset \Aut(S)$ and $ss'$ for products in $\Gamma\simeq \Z/2\Z\star \Z/2\Z \star \Z/2\Z $.

Lemma~\ref{lem:perturbed_linear_dynamics} gives the following.

\begin{pro}\label{prop: voisinage}
Let $S$ be an element of $\Fam$.
There are disjoint open neighborhoods $W_i$ of the points $p_i$ in $S$ such that given any point $q$ in one of the $W_i$, and any increasing sequence of reduced words $w_n=s_{i_{\ell(n)}} \cdots  s_{i_1}$ in $s_x$, $s_y$, $s_z$ such   that the indeterminacy point of  $s_{i_1}$ is not in $W_i$, the sequence $w_n(q)$ stays in $W_1\cup W_2\cup W_3$ and goes exponentially fast to infinity as $n$ goes to $\infty$. 

If the three letters $s_x$, $s_y$, $s_z$ appear infinitely often in the words $w_n$ as $n$ goes to $+\infty$, then the accumulation points of $(w_n(q))$ in $\Sbar$ are contained in the three vertices of the triangle $\partial S$.
\end{pro}

This proposition describes the dynamics of reduced words when $q$ is in $W_1\cup W_2\cup W_3$. Now, our goal is to describe what happens when $q$ is near $\partial S$ but is not in $W_1\cup W_2\cup W_3$. 

\begin{rem} Set $L=\{z=0=w\}\subset \partial S$ and denote by $L^*$ the open subset $L\setminus \{p_1,p_2\}$. The involutions $s_x$ and $s_y$ 
are regular on $L^*$, and both of them acts by $[x:y:0:0]\mapsto [y:x:0:0]$ on $L^*$. On the other hand, there are points arbitrary close to 
$[1:-1:0:0]$ (for instance), for which the orbit under $s_x\circ s_y$ does not remain close to $L$.
\end{rem}


Consider an increasing sequence of reduced words $w_n=
s_{i_{\ell(n)}} \cdots  s_{i_1}$ in which each of $s_x$, $s_y$, and $s_z$ ultimately appears. Let $k$ be the smallest integer such that the three involutions appear among the first $k$ letters $s_{i_j}$, $1\leq j\leq k$.
There are integers $\ell\geq 1$ and $\epsilon\in\{0, 1\}$ such that 
\begin{itemize}
\item  for $n\geq k$, $w_n$ starts with a sequence of type $s_{i_3} s_{i_1}^\epsilon (s_{i_2} s_{i_1})^\ell$, and
\item $\{i_1,i_2,i_3\}=\{x,y,z\}$.
\end{itemize}
In particular, $k=2\ell + \epsilon +1$. Then, for all $n\geq k$, the birational transformation of $\Sbar$ induced by $w_n$  contracts each edge of $\partial S$
onto some vertex of $\partial S$. Thus, if $B(p_{i_1})$ is some small neighborhood of $p_{i_1}$ in $\Sbar$,  
there exists a neighborhood $V$ of $\partial S \setminus B(p_{i_1})$ in $\Sbar$ such that $w_n(V)$ is contained in $W_1\cup W_2\cup W_3$ 
for all $n\geq k$. This neighborhood depends on $B(p_{i_1})$ and on $k$. Together with Proposition~\ref{prop: voisinage}, this proves the following corollary. 

\begin{cor}\label{cor:global_dynamics_at_infinity}
Let $w_n=s_{i_{\ell(n)}}\cdots s_{i_1}$  be an increasing sequence of reduced words in   $s_x$, $s_y$, $s_z$. 
Suppose that $w_n$ involves each of these three letters if $n$ is larger than some $n_0$.
Let $B$ be a neighborhood of the indeterminacy point $p_{i_1}$. 
There exists a neighborhood $V$ of $\partial S\setminus B$ in $\Sbar$ such that, for $n\geq n_0$, 
\begin{enumerate}
\item $w_n$ contracts $\partial S\setminus \{p_{i_1}\}$ onto one of the vertices of $\partial S$,
\item $w_n$ maps $V$ into $W_1\cup W_2\cup W_3$,
\item if $q$ is any point of $V\setminus \partial S$, its orbit $(w_n(q))$ converges towards $\partial S$ in $\Sbar$. 
\end{enumerate}
Moreover, if each of the three letters $s_x$, $s_y$, $s_z$ appears infinitely often in the sequence $(s_{i_j})_{j\geq 1}$, 
the set of accumulation points of $(w_n(q))$  is contained in $\{p_1,p_2,p_3\}$. 
\end{cor}

\begin{rem}\label{rem:global_dynamics_at_infinity}
Once we know that $(w_n(q))$ is trapped in $W_1\cup W_2\cup W_3$, we get the following: 
{\sl{if $s_{i_{\ell(n)}}$ is equal to $s_x$, then $w_n(q)$ is in $W_1$; thus, if each of the involutions appears infinitely often in the sequence of letters $s_{i_{\ell(n)}}$, the accumulation point of $(w_n(q))$ coincides with $\{p_1,p_2,p_3\}$.}}
\end{rem}

\section{From random paths to reduced words}\label{par:From random paths to reduced words}

\subsection{Cayley graph (see~\cite{Woess:DMC})}\label{par:Cayley_graph} Recall that the group $\Gamma=\langle s_x, s_y, s_z\rangle$ is a free product $\Z/2\Z\star\Z/2\Z\star\Z/2\Z$.
Let $\Graph_\Gamma$ be the Cayley graph of $\Gamma$ for the system of generators $(s_x,s_y,s_z)$: two vertices 
$g$, $h\in \Gamma$ are connected by an edge labelled $s$, for $s\in\{s_x,s_y,s_z \}$, if and only if $gs=h$.
This is a trivalent tree. The points of the boundary  $\partial \Graph_\Gamma$ are in one to one correspondance with infinite geodesic rays starting 
at the neutral element $1_\Gamma$.

Let $\mu$ be a probability measure on  $\{s_x, s_y, s_z\}$ such that $\mu(s_x)\mu(s_y)\mu(s_z)>0$.
We endow the set 
\begin{equation}
\Omega = \{s_x, s_y, s_z\}^\N
\end{equation}
with the probability measure $\mu^\N$. 

To an element $\omega = (f_0,f_1,\ldots)$ of  $\Omega$, we associate 
a path in $\Graph_\Gamma$: the path starts at $1_\Gamma$ and visits successively the vertices $f_0$, $f_0f_1$, $\ldots$, $f_0\cdots f_n$, $\ldots$. This path is typically not a geodesic (almost surely, some of
the words $f_0\cdots f_n$ are not reduced). On the other hand, with probability $1$ with respect to 
$\mu^\N$, this path goes to infinity in $\Graph_\Gamma$ (at a linear speed), and converges towards a unique point $\theta^+(\omega)$ of $\partial \Graph_\Gamma$. Starting at $1_\Gamma$ and following the geodesic ray corresponding to $\theta^+(\omega)$, one creates a sequence of vertices $(w_n)$ with $\dist(w_n,1_\Gamma)=n$. It is an increasing sequence of reduced words 
$w_n=s_{i_1} \cdots  s_{i_n}$. We shall say that $(w_n)$ is the {\bf{reduced sequence attached to}}~$\omega$. 

\subsection{Random dynamics at infinity}\label{par:random_dynamics_at_infinity}
Let us now come back to the dynamics of $\Gamma$ near $\partial S$ in 
$\Sbar$, for $S\in \Fam$. As in Section~\ref{par:Cayley_graph}, we consider a typical element $\omega$ of $\Omega$ with respect to 
$\mu^\N$; then, the reduced words 
\begin{equation}\label{eq:reduced_words_wn}
w_n=s_{i_1}\cdots s_{i_n}
\end{equation}
 converge. The inverse of $w_n$ is 
 \begin{equation}
 w_n^{-1}=s_{i_n}\cdots s_{i_1}.
 \end{equation}
We can extract a subsequence $(n_j)$ to assure that  $(w_{n_j}^{-1})$ also converges to some infinite reduced word $w^{-1}_\infty$. To do that, consider a subsequence such that the first letter of $w_n^{-1}$ is constant, then extract from it a subsequence such that the second letter is also constant, and so on.  Finally, apply a diagonal process. 

We shall now apply Corollary~\ref{cor:global_dynamics_at_infinity} to the sequence $(w_{n_j})$ (note that the indices are indexed from left to right for $w_n$ in Equation~\eqref{eq:reduced_words_wn}). We denote by $m$ an element from the sequence $n_j$.
Write
\begin{equation}
w^{-1}_\infty = s_{j_1}s_{j_2}\cdots s_{j_n}\cdots
\end{equation}
Note that we use indices $i_k$ for $w_n$ and $j_k$ for $w^{-1}_n$, and that the reduced word $w^{-1}_\infty$ provides a boundary point (and a geometric ray), denoted $\theta^-(\omega, (n_j))$. 

Let $m_0$ be the first integer such that $  s_{j_1}s_{j_2}\cdots s_{j_{m_0}}$ involves the three letters $s_x$, $s_y$, and $s_z$. 
Then, $w_m=s_{i_1}s_{i_2}s_{i_3}\cdots s_{i_k}\cdots s_{j_{m_0}}s_{j_{m_0-1}}\cdots s_{j_2}s_{j_1}$ for $m\geq m_0$.
Recall that  
$p_{j_1}$ denotes the indeterminacy point of $s_{j_1}$, viewed as a birational transformation of $\Sbar$. Let $B$ be a neighborhood 
of $p_{j_1}$. Then, Corollary~\ref{cor:global_dynamics_at_infinity} and Remark~\ref{rem:global_dynamics_at_infinity} provide a neighborhood $V$ of $\partial S\setminus B$ in $\Sbar$ (which depends on $B$ and $m_0$) such that, for $m\geq m_0$ in the sequence $(n_j)$, 
\begin{enumerate}
\item $w_m$ contracts $\partial S\setminus \{p\}$ onto the vertex $p_{i_1}$,
\item $w_m$ maps $V$ into $W_{i_1}$,
\item if $q$ is any point of $V\setminus \partial S$, its orbit $(w_{n_j}(q))$ converges towards $p_{i_1}$  in~$\Sbar$. 
\end{enumerate}
\begin{rem}\label{rk: pomega} Changing $(n_j)$, the point $p_{j_1}$ can be taken to be any of the vertices of $\partial S$. On the other hand, $p_{i_1}$ 
is uniquely determined by $\omega$.  \end{rem}

\begin{thm}\label{thm:random_dynamics_at_infinity_variation_of_q}
There is a subset $\Omega'$ of $\Omega$ of full $\mu^\N$-measure such that
for any $\omega=(f_0, \ldots, f_n, \ldots)$ in $\Omega'$, there is a vertex $p(\omega)$ of $\partial S$
with the following property. Given any vertex $q$ of $\partial S$, one can find a subsequence $(n_j)$ 
such that, for every neighborhood $B$ of
$q$, there is a neighborhood $V$ of $\partial S\setminus B$ in 
$\Sbar$ that satisfies:
\begin{enumerate}
\item $f_0\circ \cdots \circ f_{n_j}$ contracts $\partial S\setminus \{q\}$ onto $p(\omega)$;
\item
	for all $v \in V$, the sequence $f_0\circ \cdots \circ f_{n_j}(v)$
converges towards $p(\omega)$.
\end{enumerate}
\end{thm}

\begin{rem}
The sequence $(n_j)$ depends on $\omega$ and $q$, but not on $B$.
We can also impose that $w^{-1}_\infty$ starts with three distinct letters. 
In that case, $V$ depends only on (the size of) $B$.
\end{rem}

\begin{proof}
As above, we consider the point $\theta^+(\omega)\in \partial \Graph_\Gamma$ and the parametrization $(w_n)$ of the corresponding geodesic ray. The sequence $(f_0, f_1, \ldots)$ is typically not reduced (consecutive letters can be equal), but 
we can extract a subsequence $(k_j)$ to impose that (a) each of the products $f_0  \cdots f_{k_j} \in \Gamma$
is equal to one of the reduced words $w_{m_j}$ and (b) $f_0 \cdots f_{n}\neq w_{m_j}$ for all $n>k_j$. Then, we extract a further subsequence  $(n_j)$ from $(m_j)$ (or equivalently from $(k_j)$) such that $w^{-1}_{n_j}$ converges to an infinite reduced word
that starts with the letter $s_{i}$, where $i$ is the index such that $q=p_i=\Ind(s_i)$. This done, the conclusion follows from our previous discussion.
\end{proof}

To conclude this paragraph, we add a definition that will be used in Section~\ref{par:stable_manifolds}. As above, starting with a typical sequence $\omega\in \Omega$, we associate a boundary point $\theta^+(\omega)$ and a geodesic ray which starts at $1_\Gamma$. The first vertex that this
ray visits corresponds to the first letter $s_{i_1}$  of $w_n$ (see Equation~\ref{eq:reduced_words_wn}). We shall say that $s_{i_1}$
is the {\bf{initial letter}} of $\theta^+(\omega)$; it will be denoted by $\Init(\theta^+(\omega))$.\\

\section{Stationary measures}

In this section, we prove our Main Theorem (see Section~\ref{par:introduction_random_dynamics}). We fix 
a probability measure $\mu$ on  $\{s_x, s_y, s_z\}$ such that $\mu(s_x)\mu(s_y)\mu(s_z)>0$ and, as in Section~\ref{par:Cayley_graph}, 
we endow the set $\Omega = \{s_x, s_y, s_z\}^\N$ with the probability measure~$\mu^\N$. 
 For $\omega = (f_0,f_1,\ldots) \in \Omega$ and $n\in \N$, we set
\begin{equation}
f_\omega^n = f_{n-1} \circ \ldots \circ f_0 \in \Aut(S).
\end{equation}

\subsection{Stationary probability measures}
Let $\nu$ be a $\mu$-stationary probability measure on $\SC$; this means that $\nu=\mu(s_x)(s_x)_*\nu + \mu(s_y)(s_y)_*\nu + \mu(s_z)(s_z)_*\nu $ (see Equation~\eqref{eq:definition_stationary}). 
Let  $\omega$ be a typical element of  $\Omega$.  
In Sections~\ref{par:Cayley_graph} and~\ref{par:random_dynamics_at_infinity}, we introduced a boundary point $\theta^+(\omega)\in\partial\Graph_\Gamma$ and, for appropriate subsequences, boundary points $\theta^-(\omega, (n_j))\in \partial\Graph_\Gamma$. We shall use these constructions and Theorem~\ref{thm:random_dynamics_at_infinity_variation_of_q} to 
describe the support of $\nu$ and the accumulation points of stable manifolds at infinity.

\subsection{Recurrence implies compact support}\label{ss: recurrence}

Let us start with an example. Endow $\GL_2(\Z)$ with a probability measure $\mu$, the support of which is finite and generates $\GL_2(\Z)$.
The group $\GL_2(\Z)$ acts linearly on $\R^2/\Z^2$, fixes the origin $o=(0,0)$, and preserves the Lebesgue measure ${\mathrm{dvol}}=\dd x\wedge \dd y$. Moreover,  by~\cite{BFLM}, if $q\notin \Q^2/\Z^2$, then $\mu^\N$-almost every random trajectory $(f_\omega^n(q))$ equidistributes towards ${\mathrm{dvol}}$. Now, let us restrict the action to the punctured torus $X=\R^2/\Z^2\setminus \{ o\}$. If one considers the puncture as a point ``at infinity'' in $X$, 
then (a)  there is a stationary measure ${\mathrm{dvol}}$ with unbounded support, (b) $(\mu^\N\otimes {\mathrm{dvol}})$-almost every random trajectory $(f_\omega^n(q))$ is unbounded, and (c) every bounded orbit of $\GL_2(\Z)$ is finite. The following proposition excludes this type of behavior for the dynamics of $\Gamma$ on  $\SC$, for any $S\in \Fam$.

\begin{pro} \label{pro:compact_support}
Let $\nu$ be a $\mu$-stationary measure on $\SC$. Then $\nu$ has compact support.
\end{pro}

\begin{rem}\label{rem:local_fields_2}
If $\C$ is replaced by a local field $\mathbb{K}$, then the same proposition holds for $\mu$-stationary measures on  $S(\mathbb{K})$
(see Remark~\ref{rem:local_fields_1}). 
\end{rem}


\begin{proof}
{\bf{Step 1.--}} By classical results due to Furstenberg and Guivarc'h-Raugi, see \cite[Lemma 2.1, p.19]{Bougerol-Lacroix}, 
 there exists a borel subset $\Omega' \subset \Omega$ such that (a) $\mu^\N (\Omega') = 1$, (b) for every $\omega \in \Omega'$, the sequence $(f_0 \circ \ldots \circ f_{n-1})_*\nu$ converges towards a probability measure $\nu_\omega$ on $\SC$, and (c) the family of measures $(\nu_\omega)_{\omega\in\Omega'}$ satisfies 
\begin{equation}
\nu = \int_{\Omega} \nu_\omega \, d\mu^\N(\omega) . 
\end{equation}

{\bf{Step 2.--}} Now we follow the argument used by Bougerol and Picard to prove \cite[Lemma 3.3]{Bougerol-Picard}.
For every $\omega \in \Omega'$ and every increasing sequence of integers $(n_j)$, set 
\begin{equation}
H(\omega , (n_j) ) = \{ x \in \SC \; ; \; f_0 \circ \ldots \circ f_{n_j}(x) \textrm{ does not tend to }\partial S \} . 
\end{equation}
Let $\varphi : \SC \to \R^+$ be a smooth function with compact support. We have
\begin{equation}\label{eq: compsupport}
 \lim_{ j \to +\infty }\int_{\SC} \varphi (f_0 \circ \ldots \circ f_{n_j}) \, d\nu = \int_{\SC} \varphi \, d\nu_\omega .
 \end{equation}
Since $\varphi$ has compact support, we get $\varphi (f_0 \circ \ldots \circ f_{n_j}(x)) = 0$  for every $x \in \SC \setminus H(\omega , (n_j) )$ and $j$ large enough (depending on $x$). By the dominated convergence theorem, we get 
\begin{equation}
\lim_{ j \to +\infty }\int_{\SC \setminus H(\omega , (n_j) )} \varphi (f_0 \circ \ldots \circ f_{n_j}) \, d\nu = 0 . 
\end{equation}
 Equation \eqref{eq: compsupport} then implies
 \begin{equation}
  \lim_{ j \to +\infty }\int_{ H(\omega , (n_j) )} \varphi (f_0 \circ \ldots \circ f_{n_j}) \, d\nu = \int_{\SC} \varphi \, d\nu_\omega , 
  \end{equation}
which yields 
\begin{equation} \label{eq: H}
 \int_{\SC} \varphi \, d\nu_\omega  \leq \parallel \varphi \parallel_\infty \nu (H(\omega , (n_j) ) ) .  
 \end{equation}
Let $1_M : \C^3 \to [0,1]$ be a smooth function equal to $1$ on $B(0,M)$ and equal to zero on $\C^3 \setminus B(0,M+1)$. Replacing $\varphi$ in Equation \eqref{eq: H} by the restriction of $1_M$ to $\SC$ and taking the limit when $M$ tends to infinity, we obtain
$$ \nu (H(\omega , (n_j) ) = 1 $$ 
for every $\omega \in \Omega'$ and every subsequence.
In particular, the support of $\nu$ is contained in the closure of any such $H(\omega , (n_j))$. \\

{\bf{Step 3.--}} 
Fix $\omega \in \Omega'$ that satisfies also the conclusion of Theorem~\ref{thm:random_dynamics_at_infinity_variation_of_q} and  apply this theorem for two vertices $q\neq q'$ of $\partial S$. This provides two sequences $(n_j)$ and $(n_j')$ satisfying the following properties. For every neighborhood $B$ of $q$ (respectively $B'$ of $q'$), there exists a neighborhood $V$ of $\partial S \setminus B$ (respectively $V'$ of $\partial S \setminus B'$) such that for every $x \in V$ (respectively $x \in V'$), $f_0 \circ \ldots \circ f_{n_j}(x)$ (respectively $f_0 \circ \ldots \circ f_{n_j'}(x)$) tends to $p$ when $j$ tends to infinity. Since $p \in \partial S$, we get 
$\overline{H(\omega , (n_{j}) )} \cap \partial S\subset B $ and $\overline{H(\omega , (n'_{j}) )} \cap \partial S\subset B'$;
thus, 
\begin{equation}
\overline{H(\omega , (n_{j}) )} \cap \partial S \subset \{q\}
\; \; {\text{ and }} \; \;
 \overline{H(\omega , (n'_{j}) )} \cap \partial S  \subset \{q'\}.
 \end{equation} 
Since $q\neq q'$,   $\overline{H(\omega , (n_{j}) )} \cap \overline{H(\omega , (n'_{j}) )}$ is a compact subset of $S(\C)$, and by Step 2, this implies that the support of $\nu$ is compact.
\end{proof}

\subsection{Invariant probability measures} We can now strengthen Corollary~\ref{cor:classification_of_invariant_measures_I}.
 
\begin{cor}\label{cor:classification_of_invariant_measures_II} 
Let $\nu$ be a $\Gamma_S$-invariant, ergodic,  probability measure   on $S(\C)$. Then either $\nu$ is the counting measure on a finite orbit of $\Gamma_S$, 
  or the parameters $a$, $b$, $c$, and $d$ are in $[-2,2]$, $S(\R)$ has a unique compact component and $\nu$ is the symplectic measure $\nu_\R$ on $S(\R)_c$.\end{cor}
  
Indeed, if $\nu$ is invariant it is $\mu$-stationary, hence its support is compact, and the conclusion follows from Corollary~\ref{cor:classification_of_invariant_measures_I}.

\subsection{Stable manifolds}\label{par:stable_manifolds}
Now that we know that the support of any stationary measure $\nu$ is compact, we go on to conclude the proof of our Main Theorem.

Let $\sigma$ denote the one sided left shift on $\Omega$. We introduce the skew product
$ F \colon \Omega \times \SC \to  \Omega \times \SC$, defined by 
\begin{equation}
F( \omega , q ) = (\sigma(\omega) , f_0(q) )
\end{equation}
for $\omega=(f_i)_{i\geq 0}\in \Omega$ and $q\in \SC$.
The $\mu$-stationarity of $\nu$ is equivalent to the $F$-invariance of the product measure $\mu^\N \times \nu$ on $\Omega \times \SC$. 

Let us assume that $\nu$ is ergodic, and that its support $\Supp(\nu)$ is infinite. 
According to Proposition~\ref{pro:compact_support} and Theorem~\ref{thm:compact_invariant}, $S$ is defined over $\R$, $S(\R)$ has a (unique) bounded component, and
\begin{equation}
\Supp(\nu)=S(\R)_c.
\end{equation} 
Since $\Supp(\nu)$ is compact and $\Supp(\mu)$ is finite, we can apply Kingman's subadditive ergodic theorem.
This gives  two real numbers  $\lambda^+ \geq \lambda^-$ such that
\begin{equation} 
 \lambda^+ = \lim_{n \to + \infty} {1 \over n} \log \parallel D_q f^n_\omega  \parallel \ \ \ {\text{and}} \ \ \  \lambda^- = \lim_{n \to + \infty} {1 \over n} \log \parallel (D_q f^n_\omega)^{-1} \parallel^{-1}  
 \end{equation}
for $\nu$-almost every $q \in S(\R)_c$ and $\mu^\N$-almost every $\omega \in \Omega$. These real numbers are  the {\bf{Lyapunov exponents}} of $\nu$; they satisfy
\begin{equation}
\lambda^+ + \lambda^- = 0 
\end{equation}
because the smooth part of $\SC$ is endowed with the $2$-form $\area$  and $f^* \area = \pm \area $ for every $f \in \Gamma_S$ (see Section~\ref{par:invariant_forms}).

Suppose that $\lambda^+ = \lambda^- = 0$. By Ledrappier's invariance principle (see  \cite{crauel}, and \cite{Ledrappier:SaintFlour, avila-viana} for other contexts), $\nu$ is invariant under the action of $\mu$-almost every element of $\Gamma_S$, hence by $\Gamma_S$ since the support of $\mu$ generates $\Gamma_S$. Thus, our Main Theorem follows from Corollary~\ref{cor:classification_of_invariant_measures_I} in that case.

From now on, we assume that $\nu$ is {\bf{hyperbolic}}, that is 
\begin{equation} 
\lambda^+ > 0 > \lambda^-  .
\end{equation}
 For $\nu$-almost every $q \in \SC$, we define the stable manifold
\begin{equation}
 W^s_\omega (q) := \{ q' \in \SC \; ; \; \limsup_{n \to + \infty } {1\over n} \log \dist( f_\omega^n (q) , f_\omega^n (q') ) < 0 \}.
 \end{equation}

\begin{rem}\label{rk: Liou}
By \cite[Proposition 7.8]{Cantat-Dujardin:Crelle}, $W^s_\omega (q)$ is parametrized by an injective entire curve $\psi^s_{\omega,q} :\C \to \SC$. Let ${\overline {W}}^s_\omega (q)$ be the closure of $W^s_\omega (q)$ in $\Sbar$. By Liouville's theorem,  ${\overline {W}}^s_\omega (q) \cap \partial S$ is not empty.
\end{rem}

We also define 
\begin{equation}
K(\omega) = \{ q \in \SC \; ; \;  (f_\omega^n (q))_{n \geq 0}  \textrm{ is bounded in } \SC\} . 
\end{equation}
Observe that $K$ differs from the subset $H$ given in Section \ref{ss: recurrence}, since the composition of the $f_i$'s is reversed. 
We let  ${\overline{K}}(\omega)$ be the closure of $K(\omega)$ in $\overline{S}$.
  
\begin{lem}\label{lem:stable_manifolds_in_K}
For $\nu$-almost every $q$ and $\mu^\N$-almost every $\omega$, the stable manifold
$W^s_\omega(q)$ is contained in $K(\omega)$. In particular, 
$ {\overline {W}}^s_\omega (q) \cap \partial S \subset {\overline{K}}(\omega)\cap\partial S$.
\end{lem}

\begin{proof}
The orbit $\Orb_\omega(q) = \{ f_\omega^n (q) \; ; \; n \geq 0\}$ is contained in the  compact set $\Supp(\nu)$. 
Hence $\Orb_\omega(y)$ is bounded for every $y \in W^s_\omega(q)$. 
\end{proof}

\begin{pro}\label{pro:K_at_infinity}
For $(\mu^\N\otimes \nu)$-almost every $(\omega, q)$, 
\begin{enumerate}
\item the sets  ${\overline{K}}(\omega) \cap \partial S$ and ${\overline {W}}^s_\omega (q) \cap \partial S$ are equal to the indeterminacy point of the initial letter $\Init(\theta^+(\omega))$;
\item $K(\omega)$ depends on $\omega$: it is equal to one of the three vertices of $\partial S$, 
each of them being realized with positive probability. 
\end{enumerate}
\end{pro}

\begin{proof}
Write the reduced words associated to $\omega=(f_0, \ldots, f_n, \ldots)$ as $w_n=s_{i_1}\cdots s_{i_n}$, as done in Section~\ref{par:Cayley_graph}.
Choose an increasing sequence $(n_j)$ such that  the unreduced word $f_0\circ  \cdots \circ f_{n_j}$
is equal, in the group $\Gamma$, to one of the reduced words $w_{m_j}$ and $f_0\circ  \cdots \circ f_{n}\neq w_{m_j}$ for all $n>n_j$.
Let $B$ be a small neighborhood of the point $p=p_{i_1}=\Ind(s_{i_1})$, and apply Corollary~\ref{cor:global_dynamics_at_infinity}: there
is a neighborhood $V$ of $\partial S\setminus B$ such that the orbit 
\begin{equation}
f_\omega^{n_j}(q)= (s_{i_{m_j}}\circ \cdots \circ s_{i_2}\circ s_{i_1})(q)
\end{equation} of every point $q\in V$ goes to $\infty$ in $S$ as 
$n_j$ goes to $+\infty$. Thus, ${\overline{K}}(\omega)\cap \partial S$ is contained in $B$. Shrinking $B$, ${\overline{K}}(\omega)\cap \partial S\subset \{p_{i_1}\}$.
The first assertion follows from Remark \ref{rk: Liou} and Lemma~\ref{lem:stable_manifolds_in_K}.  The second assertion is a consequence of the first one and Remark~\ref{rk: pomega}.\end{proof}

The stable direction $E^s_\omega(q)$ is the tangent to $W^s_\omega(q)$ at the point $q$.  
We shall say that the stable directions define an 
invariant line field if they do not depend on $\omega$ (at least over some subset of total measure in $\Omega$). Otherwise, 
we say that they {\bf{genuinely depend}} on $\omega$. 
In fact, there are two versions of $E^s_\omega(q)$,  the complex one in $T_qS(\C)\simeq \C^2$, and the real one in $T_qS(\R)_c\simeq \R^2$;
since the complex one is the complexification of the real one, this definition does not depend on the version one chooses.

\begin{pro}\label{pro:variation_of_stable_directions}
If the ergodic, hyperbolic, stationary measure $\nu$ is not invariant, the stable directions $E^s_\omega(q)$  depend genuinely on the random trajectory $\omega=(f_0, \ldots, f_n, \ldots)$.
\end{pro}

Indeed, we only need to apply the proof of Theorem 9.1 in \cite{Cantat-Dujardin:Crelle}. 
This theorem assumes that the dynamical system is defined on a compact Kähler surface for only  two reasons. Firstly, to 
use Hodge theory and Nevanlinna currents in order to show that $W^s_\omega(q)$ depends genuinely on $\omega$; in our context, this follows from Proposition~\ref{pro:K_at_infinity}. Secondly, the compactness of the manifold is used for certain moment conditions to be satisfied; here, the same conditions are available because, according to Proposition~\ref{pro:compact_support}, the study can be done on~$S(\R)_c$.

\subsection{Classification of stationary measures: proof of the Main Theorem}
Let $S$ be an element of $\Fam$. 
Let $\nu$ be an ergodic, $\mu$-stationary, probability measure on $\SC$. As explained in Section~\ref{par:stable_manifolds}, we can assume that $\nu$ is not invariant, since otherwise our Main Theorem follows from Proposition~\ref{pro:compact_support} and Corollary~\ref{cor:classification_of_invariant_measures_I}.
Under this assumption, the  support of $\nu$ is infinite, hence equal to $S(\R)_c$, as above. 
And by Ledrappier's invariance principle we can assume that $\nu$ is hyperbolic. 
Thus, we can restrict the dynamics, i.e.\ $\Gamma_S$ and $\nu$, to $S(\R)_c$ and apply the following theorem of Brown and Rodriguez-Hertz: 
\begin{thm}[Theorem 3.4 of \cite{Brown-Rodriguez-Hertz}]
Let $M$ be a closed surface, endowed with a probability measure $\mathrm{dvol}$ given by a smooth area form. 
Let $\mu$ be a probability measure on ${\mathsf{Diff}}^2(M; \mathrm{dvol})$ with finite support; let $\Gamma_\mu$ be the subgroup of ${\mathsf{Diff}}^2(M; \mathrm{dvol})$ generated by the support of $\mu$. Let $\nu$ be a hyperbolic, ergodic, and $\mu$-stationary probability measure on $M$. If for $\nu$-almost every $q\in M$ the stable direction $E^s_\omega(q)$ depends genuinely on $\omega$, then $\nu$ is
invariant by $\Gamma_\mu$. 
\end{thm}
By Proposition~\ref{pro:variation_of_stable_directions}, the stable directions $E^s_\omega(q)$ depend genuinely on $\omega$ for $\nu$-almost every $q$. By this Theorem, $\nu$ should be invariant, contradiction. 

\begin{rem}
The surface $S(\R)_c$, in our case, may have singularities. But they  are quotient singularities and the group $\Gamma_S$ preserves the orbifold structure of $S(\R)$, thus \cite{Brown-Rodriguez-Hertz} can be applied without any change, even if $S(\R)$ is singular. 
\end{rem}

\section{Margulis functions and applications: an example}

In this final section, we study the example from Section~\ref{par:boalch-klein}, the goal being to describe further properties satisfied by the dynamics of $\Gamma$.

\subsection{The surface} The parameters are $(A,B,C,D)=(1,1,1,0)$. We write $S$ for $S_{(1,1,1,0)}$ and $\Gamma$ for $\Gamma_S$. The surface $S$ is smooth and the compact component $S(\R)_c$ is homeomorphic to a sphere. According to~\cite{Lisovyy-Tykhyy}, every orbit of $\Gamma$ in $\SC$ is infinite, except the orbit of the origin $o=(0,0,0)$. 
Thus,  our main theorem shows that the space of stationary measures is an interval, the endpoints of which are  
\begin{itemize}
\item   the symplectic measure $\nu_\R$, supported by $S(\R)_c$,  and
\item $\delta_{\Gamma(o)}$, the counting measure on the finite orbit.   
\end{itemize}

We will be interested in the following problem. Fix a point $q$ of $\SC \setminus \Gamma(o)$. Given $\omega=(f_0, \ldots, f_n, \ldots)\in \Omega$, consider the empirical measures $\nu_N(\omega; q)$ defined in Equation~\eqref{eq:empirical_measures}. As already said in Section~\ref{par:introduction_random_dynamics}, for a typical $\omega$ any cluster value $\nu$ of $(\nu_N(\omega; q))$ in the space of measures on $\SC$ is a stationary measure (though its total mass may be $<1$).
The question is to determine the decomposition  of such a measure $\nu$ as a convex combination $\alpha \delta_{\Gamma(o)}+\beta \nu_\R$.

\subsection{Expansion along $\Gamma(o)$} The stabilizer of $o$ is a subgroup of $\Gamma$ of index $7$ that contains $f=(s_y\circ s_x)^2$, $g=(s_x\circ s_z)^2$, $h=(s_z\circ s_y)^2$.
The tangent space at the origin is given by the equation $u+v+w$ and can be parametrized by $(u,v,-u-v)$. In these coordinates,  the differentials of $f$, $g$, and $h$ at the origin act on $T_oS$ by
\begin{align}
Df_o(u,v) &= (2u+v, -u), \\ 
Dg_o(u,v) &= (u+v, v), \\
Dh_o(u,v) &= (u, -u+v).
\end{align}
Thus, $Df_o$, $Dg_o$, and $Dh_o$ generate the group $\SL_2(\Z)\subset \GL(T_oS)$.

According to~\cite[Theorem 8.16]{cantat-dujardin:expansion}, this implies that the finite orbit $\Gamma(o)$ is expanding, in the sense of~\cite[Section 1.3]{cantat-dujardin:expansion} and~\cite{Chung}; in other words, in average, the dynamics of $\Gamma$ is infinitesimally repelling along the orbit $\Gamma(o)$. From~\cite[Theorem 4.3]{cantat-dujardin:expansion}\footnote{In this statement and in Theorem 4.4 of \cite{cantat-dujardin:expansion}, the ambient complex manifold is assumed to be compact, but this is not really used in the proof. }), we get:\\ 
\indent {\sl{if $q$ is a point of $S(\R)_c\setminus \Gamma(o)$, then for a typical $\omega$ in $\Omega$, the empirical 

measures $\nu_N(\omega; q)$ converge towards the symplectic measure $\nu_\R$.}} \\
In other words, as soon as $q\in S(\R)_c$ is not on the finite orbit, then its typical random trajectories do not charge $\Gamma(o)$.

\subsection{Expansion along $S(\R)_c$} The transformation $f$ from the previous paragraph act on $S$ in the following way. It preserves the fibration $\pi_3\colon S(\C)\to \C$, $\pi_3(x,y,z)=z$. It acts by homography on each fiber $x^2+y^2+z_0^2+xyz_0=x+y+z_0$. Along $S(\R)_c$, $\pi_3$ is a fibration in (topological) circles (except for the singular fibers) along which $f$ acts as a rotation by an angle that depends analytically on $z_0$. The behavior of $g$ and $h$ is similar, but with respect to the other fibrations $\pi_2$ and $\pi_1$. 

Thus, the situation is slightly different from~\cite{cantat-dujardin:expansion}, in which parabolic automorphisms preserve fibrations into curves of genus $1$, but the same arguments apply. From~\cite[Theorem 1.5]{cantat-dujardin:expansion}, we deduce that the dynamics is expanding along $S(\R)_c$. Then, \cite[Theorem 4.5]{cantat-dujardin:expansion} shows that, using the complex structure, the expansion can be transfered transversally to $S(\R)_c$, and this shows that \\
\indent {\sl{if $q$ is a point of $\SC\setminus S(\R)_c$ and $\omega$ is typical, then the only cluster value 

of $(\nu_N(\omega; q))$ is the zero measure: in average, the orbit of $q$ goes to $\partial S$.}} 
\vspace{0.1cm}

{\noindent}This argument also shows the following. 

\indent {\sl{Suppose that $S\in \Fam$ is defined over $\R$, $S(\R)$ has a smooth compact compo-

nent $S(\R)_c$, and $\SC$ does not contain any finite orbit. Let $q$ be a point from

$\SC\setminus S(\R)$. Then, for a typical random sequence $\omega$,  the only cluster value 

of $(\nu_N(\omega; q))$ is the zero measure.}}   

\section{Appendix}

{\small{
We extend the study of Section~\ref{par:From random paths to reduced words}  to relate it  to hyperbolic geometry and Furstenberg theory and to improve Proposition~\ref{pro:compact_support}.

\subsection{Isometries of $\H$} The group $\Gamma\simeq \Gamma_2^{\pm}$ can also be viewed as a group of isometries of the upper half plane $\H$. Explicitly, the generators $s_x$, $s_y$, $s_z$ are mapped to the involutive isometries (see Section~\ref{par:vieta_involutions})
\begin{equation}
 \sigma_x(z) = - \bar z +2, \ \ \sigma_y(z) = {\bar z \over 2 \bar z -1}, \ \ \sigma_z(z) = -{ \bar z} .
 \end{equation}
Each involution is the reflection with respect to one side of the ideal triangle with vertices 
$0$, $1$, $\infty$ in $\partial\H=\R\cup\{\infty\}$.
This triangle $\T$ is a fundamental domain of the action of $ \Gamma_2^{\pm}$, its images tesselate $\H$, and the dual graph of this tesselation 
can be identified to $\Graph_\Gamma$.

\begin{rem} The Cayley map $(z-\ii)/(z+\ii)$ maps the upper half plane to the unit disk $\disk$ and the triangle to the ideal triangle $\T_\disk\subset \disk$ with vertices $1$, $-1$ and $-i$. 
\end{rem}

As in Section~\ref{par:Cayley_graph}, let $\omega$ be a typical element of $\Omega$ and let $(w_n)$ be the sequence of reduced words derived from $\omega$, 
with $w_n=s_{i_1} \cdots  s_{i_n}$. The isometries 
$\sigma_{i_1}\circ \cdots \circ \sigma_{i_n}$ map $\T$ to a sequence of adjacent triangles 
\begin{equation}
\T_0=\T, \; \; \T_1=\sigma_{i_1}(T), \; \; \T_2=\sigma_{i_1}(\sigma_{i_2}(T)), \; \; \ldots.
\end{equation} the dual 
Then, the sequence of triangles $(\T_n)$ 
converges in the Hausdorff topology of $\overline\H$ towards a unique boundary point $\theta^+_\H(\omega)$.

\begin{figure}[h]
\includegraphics[width=0.6\textwidth]{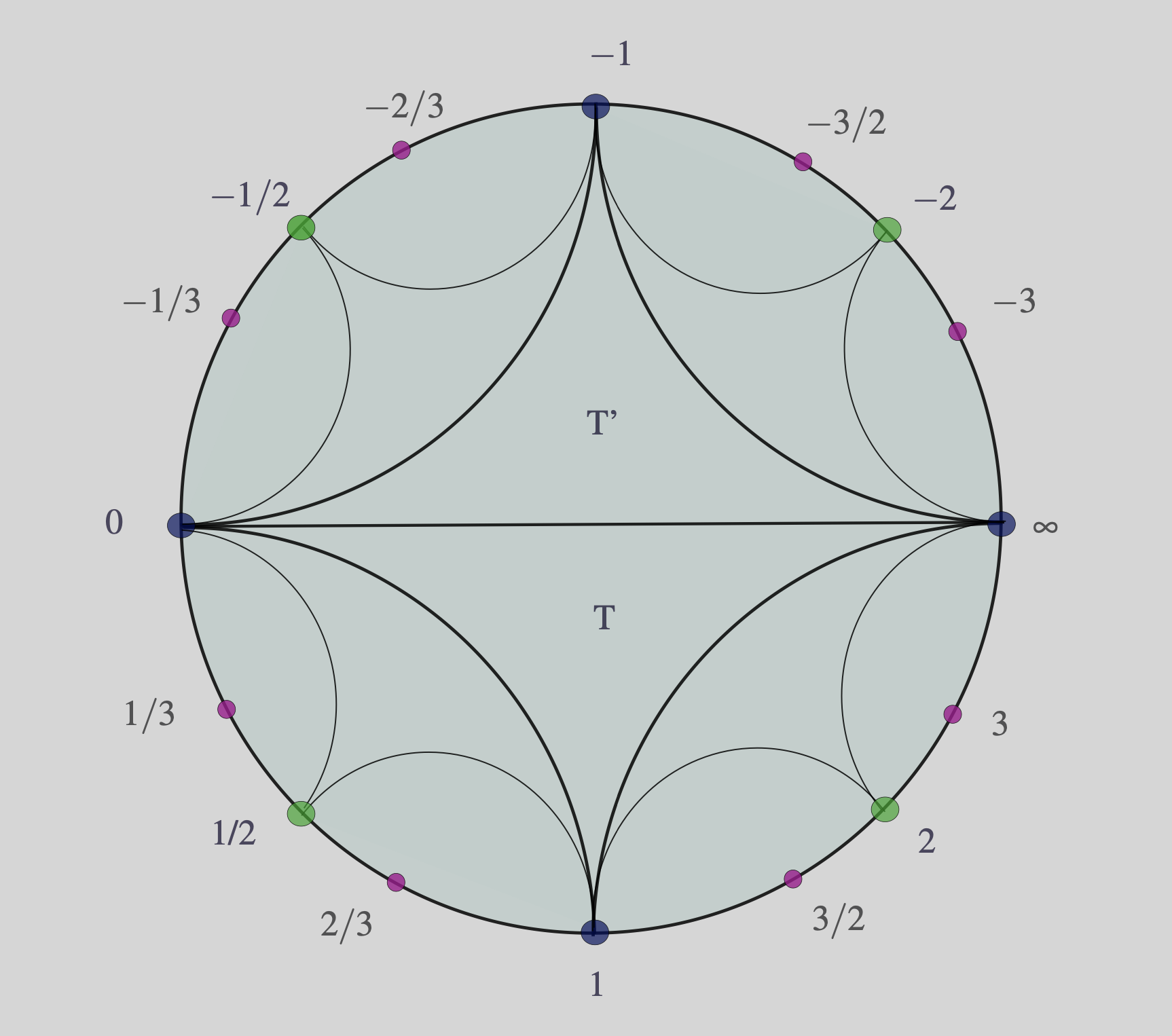}
\caption{
{\small{ We draw a picture in the unit disk, using the change of variable $z\mapsto  (z-\ii)/(z+\ii)$, but we label boundary points according to the natural parametrization $\partial \H=\R\cup\{\infty \}$.
The triangle $\T$ and its images under the action of $\Gamma_2^\pm$ tesselate the hyperbolic disk. For instance, the triangle $\T'$  
corresponds to $\sigma_z(\T)$, and the magenta points correspond to curves of depth $2$ in ${\overline{S}}_2$. 
}}
}
\end{figure}

\subsection{Matrices and Furstenberg theory (see~\cite{Bougerol-Lacroix})}\label{par:basic_Furstenberg_theory} Fix a norm $\parallel \cdot \parallel$ on $\Mat_2(\R)$, and we recall that the boundary $\partial \H$ is naturally identified to the projective line $\P^1(\R)=\P(\R^2)$. 
The isometries $\sigma_x$, $\sigma_y$, $\sigma_z$ correspond to elements $\hat\sigma_x$, $\hat\sigma_y$, $\hat\sigma_z$ of $\PGL_2(\R)$ (see Equation~\eqref{eq:involution_matrices}). We lift these elements to matrices in $\GL_2(\R)$, using the same notation.
Typically, the norm of the product $\hat\sigma_{i_1} \cdots  \hat\sigma_{i_n}$ 
goes exponentially fast to $+\infty$. If we normalize $\hat\sigma_{i_1}  \cdots  \hat\sigma_{i_n}$ by dividing by its norm, 
we obtain a sequence of matrices 
\begin{equation}
[\hat\sigma_{i_1} \cdots  \hat\sigma_{i_n}]=\frac{\hat\sigma_{i_1} \cdots  \hat\sigma_{i_n}}{\parallel \hat\sigma_{i_1} \cdots  \hat\sigma_{i_n}\parallel}\in \Mat_2(\R),
\end{equation}
each of which has norm $1$. Extracting a subsequence $(n_j)$, we may assume that this sequence converges towards an element $\hat\sigma_\infty$ of $\Mat_2(\R)$; then, with probability $1$ with respect to $\mu^\N$, we have 
\begin{enumerate}
\item $\hat\sigma_\infty$ has rank $1$ and its image coincides with the line corresponding to $\theta^+_\H (\omega)$ in the identification 
$\partial \H=\P^1(\R)$; in particular, the image does not depend on the subsequence.
\end{enumerate}
On the other hand the kernel does depend on $(n_j)$. More precisely, consider the inverse of $s_{i_1} \cdots  s_{i_n}$, 
which is the same thing as the reversed word $s_{i_n} \cdots  s_{i_1}$. Then, as in Section~\ref{par:random_dynamics_at_infinity}, 
extract a subsequence such that each of 
$s_{i_{n_j}}$, $s_{i_{n_j-1}}$, etc,  converges towards an element of $\{s_x,s_y,s_z\}$; in other words, extract a subsequence to 
make the path $s_{i_{n_j}}$, $s_{i_{n_j}} s_{i_{n_j}-1}$, etc, converge into $\Graph_\Gamma$ to a geodesic ray. This ray corresponds 
to a point of $\partial \Graph_\Gamma$, to a sequence of adjacent triangles in $\H$, and to a boundary point $\theta^-_\H(\omega, (n_j))
\in \partial H= \P^1(\R)$. With such a choice, 
\begin{enumerate}
\item[(2)] the kernel of $\hat\sigma_\infty$ is the line in $\R^2$ corresponding to the point $$\theta^-_\H(\omega, (n_j))\in \P^1(\R).$$ 
\end{enumerate}
For a typical $\omega$, each point of $\P^1(\R)$ is equal to $\theta^{-}_\H(\omega, (n_j))$ for some subsequence $(n_j)$. 
This follows from two facts: (a) the limit set $\Lim(\Gamma)\subset \partial \H$ is equal to $\H$ and (b) the support of the Furstenberg measure 
is equal to $\Lim(\Gamma)$.

\begin{rem} Assume that $(n_j)$ is chosen in such a way that $\theta^+_\H(\omega)\neq \theta^-_\H(\omega, (n_j))$
Let $I\subset \partial \H$ be an open neighborhood of $\theta^-_\H(\omega, (n_j))$. If $I$ is small enough, its complement $I^{c}=\partial \H\setminus I$ is a compact neighborhood of $\theta^+_\H(\omega)$, and in the Hausdorff topology of $\partial \H$, 
the sequence $\sigma_{i_1}\circ \cdots \circ \sigma_{i_n}(I^c)$ converges towards $\theta^+_\H(\omega)$.
\end{rem}

\subsection{Compactifications of $S$} \label{par:compactifications}

\subsubsection{} Let $m$ be a positive integer. Blow-up $\overline{S}$ at the vertices of $\partial S$ to get a compactification by $6$ rational curves organized in a hexagon, and repeat this process to get a compactification ${\overline{S}}_m$ of $S$ by $N:=3\cdot 2^m$ rational curves organized in a cycle. We denote by $\partial S_m$ its boundary.

\subsubsection{} The dual graph ${\mathcal{C}}_m$ of $\partial S_m$ has one vertex per rational curve, and one edge between two vertices if the corresponding curves have a point in common.  Topologically, ${\mathcal{C}}_m$ is a circle with $N$ points on it. 
The graph ${\mathcal{C}}_m$ is obtained from ${\mathcal{C}}_{m-1}$ by adding a new vertex in the middle of each edge, i.e.\ by doing a barycentric subdivision. We say that a vertex of ${\mathcal{C}}_m$ has depth $k\leq m$ if it appears first in the $k$-th barycentric subdivision ${\mathcal{C}}_k$. The three lines of the triangle $\partial S_0=\partial S$ correspond to vertices of depth $0$.
One way to parametrize ${\mathcal{C}}_m$ is the following. In $\partial \H$, 
consider the three end points of $\T$ and their images under the action of all reduced words in $\{ \sigma_x, \sigma_y, \sigma_z\}$ of length $\leq m$; we obtain a circle $\partial \H$ with $N$ marked points, hence a graph ${\mathcal{H}}_m$; the edges of ${\mathcal{H}}_m$ are intervals of $\partial \H$. There is an equivariant map sending $\partial \H$ to ${\mathcal{C}}_m$: it maps the vertices of $\T$ to the vertices of ${\mathcal{C}}_m$ of depth $0$, then the (new) vertices in $\sigma_x(\T)$, $\sigma_y(\T)$, $\sigma_z(\T)$, etc.

\subsubsection{} Now, consider an element $f$ in $\Gamma\simeq \Gamma_2^{\pm}$.  Suppose that the $3$ vertices of $f^{-1}(\T)$ are contained in the interior of an edge $I$ of ${\mathcal{H}}_m$ and that the three vertices of $f^{-1}(\T)$ are contained in the interior of an edge $J$. These edges correspond to edges $I'$ and $J'$ of ${\mathcal{C}}_m$, hence to vertices $p$ and $q$ of the cycle of rational curves $\partial S_m$. Then the birational map induced by $f$ on ${\overline{S}}_m$ contracts $\partial S_m\setminus \{q\}$ onto $p$.
Moreover, one can construct neighborhoods $W_j$ of the vertices of $\partial S_m$ satisfying properties which are analogous to the ones listed in Corollary~\ref{cor:global_dynamics_at_infinity} and then extend Theorem~\ref{thm:random_dynamics_at_infinity_variation_of_q}.
We do not prove these properties. But they can be derived from~\cite{Nguyen-Bac-Dang:M2} (or~\cite[Chapter 8]{Cantat-Cornulier}) and the computations done in Sections~\ref{par:Dynamics at infinity} and~\ref{par:From random paths to reduced words}. Nguyen Bac Dang studies the dynamics of $\Gamma$ on the space of valuations centered at infinity in $S$.  

\subsection{Application} 
One can now strengthen Proposition~\ref{pro:compact_support} as follows. 

\begin{thm}
Let  $S_1$, $S_2$, $\ldots$, $S_m$ be  elements of $\Fam$. If $\nu$ is a stationary measure for the diagonal action of $\Gamma$ on 
$S_1(\C)\times \cdots \times S_m(\C)$, the support of $\nu$ is compact. 
\end{thm}

\begin{proof}[Sketch of Proof]
Set $M=S_1\times \cdots \times S_m$ and denote by $\overline{M}$ its closure in $(\P^3)^m$. The boundary $\partial M={\overline{M}}\setminus M$  is the union of the sets 
\begin{equation}
\partial_i M=\overline S_1 \times \cdots \times \partial S_i\times \cdots \times \overline S_m.
\end{equation}
Define $H(\omega, (n_i))$ as in the proof of Proposition~\ref{pro:compact_support}, but for the diagonal dynamics of $\Gamma$ on $M$. 

First, assume $m=2$. The closure of $H(\omega, (n_i))$ in $\overline{M}$ intersects
its boundary on a set that is contained in ${\overline{S_1}}\times \{q\}\cup \{q\}\times {\overline{S_2}}$; varying the choice of the subsequence $(n_i)$, the point $q$ can be chosen to be any of  the vertices at infinity. Now, if $q\neq q'$,
\begin{equation}
({\overline{S_1}}\times \{q\}\cup \{q\}\times {\overline{S_2}}) \cap ({\overline{S_1}}\times \{q'\}\cup \{q'\}\times {\overline{S_2}}) 
= \{(q,q'), (q',q)\}.
\end{equation}
Thus, the accumulation points of the support of $\nu$  in $\partial M$ are contained in $\{(q,q'), (q',q)\}$ for any pair of vertices $(q,q')$ at infinity. Since 
\begin{equation}
\{(q,q'), (q',q)\}\cap \{(q,q''), (q'',q)\}=\emptyset
\end{equation}
when $q$, $q'$, $q''$ are pairwise distinct, this shows that the support of $\nu$ is compact.

Now, let us prove the result for any $m\geq 1$. For this we use Section~\ref{par:compactifications}. Blow-up each $\overline{S_i}$ at the vertices of $\partial S_i$ to get a compactification by $6$ rational curves organized in a hexagon, and repeat this process $m$ times to get a compactification by $N:=3\cdot 2^m$ rational curves organized in a cycle; let $q_i$ be the vertices of this cycle, with $1\leq i\leq N$. Then, choosing correctly $(n_i)$, one sees that the accumulation points of $\Supp(\nu)$ in $\partial M$ are contained in  
\begin{equation}
B(I)=\{(p_1, \ldots, p_m)\; ; \; p_i\in\{q_{i_1}, \ldots, q_{i_m}\} \; \; {\text{ for each}} \; \;  i\leq m\}
\end{equation}
for each multi-index $I=(i_1, \ldots, i_m)\in \{1, \ldots, N\}^m$. The intersection of the $B(I)$ being empty, $\Supp(\nu)$ is compact. 
 \end{proof}

}}

\bibliographystyle{plain}
\bibliography{referencesMarkov}

\end{document}